\documentclass[11pt,a4paper]{amsart}

\usepackage[toc,page]{appendix}
\usepackage{lscape}
\usepackage{hyperref} 
\usepackage{multirow,arydshln}
\usepackage[justification=centering]{caption}

\usepackage{amsmath,amsthm, amscd, amssymb, amsfonts}
\usepackage{epsfig}
\usepackage{amsmath,amsthm,amssymb, amscd,enumerate}

\numberwithin{equation}{section}

\theoremstyle{plain}
\newtheorem{theorem} {Theorem} [section]
\newtheorem*{theoremSN} {Theorem A}
\newtheorem*{theoremSNB} {Theorem B}
\newtheorem{lemma} [theorem] {Lemma}
\newtheorem{corollary} [theorem] {Corollary}
\newtheorem{proposition} [theorem] {Proposition}

\theoremstyle{definition}

\newtheorem{remark} [theorem] {Remark}

\newtheorem*{acknow} {Acknowledgments}

\renewcommand \parallel {/\kern-3pt/}

\newcommand \g {\mathfrak{g}}

\newcommand \h {\mathfrak{h}}

\renewcommand \k {\mathrm{k}}

  {\centering\normalfont\Large\bfseries}

\begin{document}


\title{Faithful Representations of Minimal Dimension of $6$-dimensional nilpotent Lie algebras}

\title[Faithful Representations of the Nilpotent Lie Algebra]{Faithful Representations of Minimal Dimension of $6$-dimensional nilpotent Lie algebras}

\author{Nadina Elizabeth Rojas}

\address{Current affiliation: CIEM, FaMAF, FCEFyN Universidad Nacional de C\'ordoba, \newline \indent Ciudad Universitaria, \newline \indent
(5000) C\'ordoba, \newline \indent Argentina}

\email{nrojas@efn.unc.edu.ar}

\thanks{Fully supported by CONICET-FaMAF-FCEFyN (Argentina)}



\begin{abstract}
The main goal of this paper is to compute
$\mu(\g)$ and
$\mu_{nil}(\g)$ for each nilpotent Lie algebra
$\g$ of dimension
$6$ over a field of characteristic zero
$\k$. Here
$\mu(\g)$ and
$\mu_{nil}(\g)$ is the minimal dimension of a faithful representation of
$\g$ and the minimal dimension of a faithful nilrepresentation of
$\g$, respectively.
We also give a minimal faithful representation and nilrepresentation for each nilpotent Lie algebra of dimension
$\leq 6$.
\end{abstract}

\maketitle


\section{Introduction}\label{intro}


In this paper all Lie algebras and representations are finite dimensional over a field
$\k$ of characteristic zero. Ado's theorem asserts that every Lie algebra has a finite dimensional faithful representation
(see for instance \cite[page 202]{J}).
Given a Lie algebra, let
$$
\mu(\g) = \min \{\dim V : (\pi, V) \text{ is a faithful representation of } \g\},
$$
similarly, let
$$
\mu_{nil}(\g) = \min \{\dim V : (\pi_{nil}, V) \text{ is a faithful nilrepresentation of } \g\}.
$$
It is clear that
$\mu(\g)$ and
$\mu_{nil}(\g)$ are an invariant integer of
$\g$ and
$\mu \leq \mu_{nil}$.
In the theory of Lie algebras, there is a interest in the value of
$\mu(\g)$ of a Lie algebra
$\g$. For instance,
$\mu(\g)$ plays an importance role in the theory of the affine crystallographic groups and
finitely-generated torsion free nilpotent groups.
It is known that for a Lie groups
$G$, if  admits a left-invariant affine structure if and only if
$\g= Lie(G)$ (the Lie algebra of
$G$) admits a left-symmetric algebra structure (see \cite{Se} or \cite{K}).

By the other hand, if
$\g$ admits a left-symmetric algebra structure, this one induces a faithful affine representation
$\alpha : \g \rightarrow \mathfrak{aff}(\g)$, called the \emph{affine holonomy representation} (see \cite{FGH}). If
$\dim \g = n$ then
$\mathfrak{aff}(\g) \subseteq \mathfrak{gl}(n+1)$ and we obtain a faithful representation of the Lie algebra
$\g$. It follows that
$$
\mu(\g) \leq n + 1.
$$

In this area, a remarkable example of a Lie group
$G$ not admitting any invariant affine structure was given by Yves Benoist in \cite{Be}.
He constructed a nilpotent Lie algebra
$\g$ such that
$\mu(\g) > \dim \g + 1$, this result is a counterexample
of a well now Milnor conjecture \cite{Mi} (see also \cite{BG}).

In general, there are no methods to determine these invariants for a given Lie algebra.
Moreover, it seems to be very hard to obtain
$\mu(\g)$ for certains Lie algebra
$\g$, for instance the free
$2$-step nilpotent Lie algebra. The
$\mu$ has been achieved for very few
families (see for instance \cite{Be}, \cite{BM}, \cite{CR}, \cite{KB}, \cite{J2}, \cite{R}, \cite{S}).

Let
$\g$ be a nilpotent Lie algebra, it is known that
$\mu(\g) < \frac{3}{\sqrt{\dim \g}} 2^{\dim \g}$
(see \cite{B}). In particular, if
$\g$ is a nilpotent Lie algebra of dimension
$\leq 7$, we get
$\g$ admit a left-symmetric algebra structure then
\begin{equation}\label{eq:5}
\mu(\g) \leq \dim \g + 1.
\end{equation}
This result was proved by Y. Benoist (see \cite[Lemma 6]{Be}).

By the other hand, the nilpotent Lie algebras of dimension
$6$ have been classified by different authors (see for instance
\cite{Mo}, \cite{Ni}). In particular, De Graaf \cite{Gr} give
a classification over a fields of characteristic different from
$2$. If
$\g$ is a nilpotent Lie algebra of dimension
$6$, by equation (\ref{eq:5}) we have
$\mu(\g) \leq 7$.

This problem is also related
to the representation theory of finitely generated nilpotent groups, since it is
important to have methods for finding representations by integer matrices
of small dimension for this class of groups (see for instance \cite{GN}, \cite{N}).

Our main result in this paper are the value of
$\mu$ and
$\mu_{nil}$ for all nilpotent Lie algebra of dimension
$6$ over a field
$\k$ of characteristic zero.
Furthermore, we obtain faithful representations and faithful nilrepresentations of
minimal dimension for all nilpotent Lie algebras of dimension
$\leq 6$.

Throughout this paper we use De Graaf's classification given in \cite{Gr} and
in this direction, we prove the following theorem.

\begin{theoremSN}
According to the notation used in \cite{Gr}, the values of
$\mu$ and
$\mu_{nil}$ for the Lie algebras of dimension 6 over a field
$\k$ of characteristic zero are
\begin{table}[!h]
\begin{centering}
$$
\begin{tabular}{||l||c|c||}
\hline
\multicolumn{1}{||c||}{\bf{ De Graaf's classification}} & $\pmb{\mu}$ & $\pmb{\mu}_{nil}$\\
\hline
\hline
\vbox to .4cm{}
$L_{6,19}(\epsilon)$; if there exists $\alpha \in \k^*$ such that $\alpha^2 + \epsilon = 0$.        & 4 & 4 \\
\hline
\vbox to .4cm{}
$L_{6,3}$,  $L_{6,4}$,  $L_{6,5}$,  $L_{6,8}$                         & 4 & 5 \\
\hline
\vbox to .4cm{}
$L_{6,1}$,  $L_{6,2}$,  $L_{6,6}$,  $L_{6,7}$, $L_{6,10}$, $L_{6,11}$,$L_{6,12}$, $L_{6,13}$ &\multirow{2}{*}{5} & \multirow{2}{*}{5}\\
\vbox to .4cm{}
$L_{6,20}$, $L_{6,21}(\epsilon)$, $L_{6,22}(\epsilon)$   $L_{6,23}$, $L_{6,25}$, $L_{6,26}$     &                   &                   \\
\hline
\vbox to .4cm{}
$L_{6,24}(\epsilon)$; if there exist $\alpha \in \k$ such that  $\alpha^2 - \epsilon = 0$.           & 5 & 5  \\
\hline
\vbox to .4cm{}
$L_{6,19}(\epsilon)$; if $\alpha^2 + \epsilon \neq 0$ for all $\alpha \in \k^*$. & 5 & 5\\
\hline
\vbox to .4cm{}
$L_{6,9}$                                                       &5& 6 \\
\hline
\vbox to .4cm{}
$L_{6,14}$, $L_{6,15}$, $L_{6,16}$, $L_{6,17}$, $L_{6,18}$              &6& 6 \\
\hline
\vbox to .4cm{}
$L_{6,24}(\epsilon)$; if $\alpha^2 - \epsilon \neq 0$   for all $\alpha \in \k$.    & 6 & 6  \\
\hline
\end{tabular}
$$
\caption{$\mu$ and $\mu_{nil}$ for all nilpotent Lie algebras of $dim= 6$}
\label{Tabla:1}
\end{centering}
\end{table}
\end{theoremSN}

The paper is organized as follows. In \S \ref{Sec:22}, a minimal faithful representation and a
faithful nilrepresentation are given for all nilpotent Lie algebra of dimension
$\leq 5$. In this direction, we prove the following theorem:

\begin{theoremSNB}\label{dim5}
According to the notation used in \cite{Gr}, the values of
$\mu$ and
$\mu_{nil}$ for the Lie algebras of dimension 5 are
\begin{table}[!h]
\begin{centering}
$$
\begin{tabular}{||l||c|c||}
\hline
\multicolumn{1}{||c||}{{\bf Lie algebra}} & $\pmb{\mu}$ & $\pmb{\mu}_{nil}$\\
\hline
\hline
\vbox to .4cm{}
$L_{5,3}$, $L_{5,4}$, $L_{5,5}$, $L_{5,8}$     & 4         & 4 \\
\hline
\vbox to .4cm{}
$L_{5,1}$                         & 4         & 5 \\
\hline
\vbox to .4cm{}
$L_{5,2}$, $L_{5,7}$,  $L_{5,6}$, $L_{5,9}$  & 5         & 5 \\
\hline
\end{tabular}
$$
\caption{$\mu$ and $\mu_{nil}$ for all nilpotent\\ Lie algebras of $\dim= 5$}
\label{Tabla:2}
\end{centering}
\end{table}
\end{theoremSNB}
The Theorem B is an extension of results that have been obtained independently in \cite{BNT}.

In \S \ref{sec:dim6}, the upper bound for
$\mu(L_{6,j})$ is obtained by an explicit faithful representation and faithful nilrepresentation of minimal dimension for all
$j = 1, \dots, 24$ and we give justification for the lower
bound for the Lie algebras
$L_{6,j}$ with
$j \neq 9, 19, 24$.

In \S \ref{sec:L619} we proof
$\mu(L_{6,19}(\epsilon))$ for every
$\epsilon \in \k$ and in \S \ref{sec:L6924} we prove the lower bound for
$\mu_{nil}(L_{6,9})$ and
$\mu_{nil}(L_{6,24}(\epsilon))$ for every
$\epsilon \in \k$.


\section{Preliminaries}


Let us mention some results on finite dimensional representations of nilpotent Lie algebras
that will be needed in the following section.

It is a well known theorem of Zassenhaus.

\begin{theorem}\cite[page 41]{J}
Let
$\g$ be a finite dimensional nilpotent Lie algebra and let
$(\pi, V)$ be a finite dimensional representation of
$\g$. If
$\k$ is algebraically closed then
$$
V= V_1 \oplus \dots \oplus V_s
$$
such that
$\pi(X)\mid_{V_i}$ is a scalar
$\lambda_i(X)$ plus a nilpotent operator
$\pi_{N_i}(X)$ on
$V_i$ for all
$X \in \g$ and
$i=1, \dots, s$. Moreover,
$(\pi_{N_i}, V_i)$ is a nilrepresentation of
$\g$ for all
$i=1, \dots, s$.
\end{theorem}

An immediate consequence of this result is that every representation
of a nilpotent Lie algebra has a Jordan decomposition (see for instance \cite[Theorem 2.1 and Theorem 2.2]{CR}).

\begin{proposition}\label{coro:nilpot}
Let
$\g$ be a finite dimensional nilpotent Lie algebra and let
$(\pi, V)$ be a finite dimensional representation of
$\g$. For each
$X \in \g$, let
$\pi_S(X)$ and
$\pi_N(X)$ be, respectively, the semisimple and nilpotent parts of the additive Jordan decomposition of
$\pi(X)$. Then
$(\pi_S, V)$ and
$(\pi_N, V)$ are representation of
$\g$. Moreover, if the center
$\mathfrak{z}(\g)$ is contained in
$[\g, \g]$ then
$(\pi, V)$ is faithful if and only if
$(\pi_N, V)$ is faithful.
\end{proposition}

\begin{corollary}\label{coro:mumunil}
Let
$\g$ be a finite dimensional nilpotent Lie algebra such that the center
$\mathfrak{z}(\g)$ is contained in
$[\g, \g]$ then
$\mu(\g) = \mu_{nil}(\g)$.
\end{corollary}

If
$\g$ is a nilpotent Lie algebra and
$(\pi_{nil}, V)$ is a nilrepresentation of
$\g$, we define the linear mapping
$\pi : \g \oplus \k \rightarrow \mathfrak{gl}(V)$ given by
$\pi(X, a)= \pi_{nil}(X) + aI$ where
$I$ is the identity map on
$V$. Since
$(\pi_{nil}, V)$ is a representation of
$\g$, by a straightforward calculation we have
$(\pi, V)$ is a representation of
$\g \oplus \k$ and if
$(\pi_{nil}, V)$ is a faithful then
$(\pi, V)$ is faithful too. Thus
$\mu(\g \oplus \k) \leq \mu_{nil}(\g)$.

\begin{proposition}\label{munilabeliana}
Let
$\g$ be a Lie algebra of finite dimension over a field
$\k$ of characteristic zero.
\begin{enumerate}[(a)]
\item \label{abelian}
      If $\g$ is an abelian Lie algebra then
      $\mu(\g) = \left\lceil 2\sqrt{\dim \g - 1} \right\rceil$ and
      $\mu_{nil}(\g) = \left\lceil 2\sqrt{\dim \g} \right\rceil$ (see for instance \cite{S}, \cite{J2}, \cite{MM}).
\item \label{filiform}
      If $\g$ is a filiform Lie algebra then
      $\mu(\g) \geq \dim \g$ and if
      $\dim \g < 10$ then
      $\mu(\g) = \dim \g$ (see \cite{B}). Also
      $\mu_{nil}(\g) = \mu(\g)$ (see Proposition \ref{coro:nilpot}).
\item \label{directsum}
      If $\g$ is a nilpotent Lie algebra then
      $\mu(\g \oplus \k) \leq \mu_{nil}(\g)$ and if
      $\mathfrak{z}(\g) \subseteq [\g , \g]$ then
      $\mu(\g \oplus \k) = \mu(\g)$ (see Corollary \ref{coro:mumunil}).
\end{enumerate}
\end{proposition}


\begin{section}{Minimal faithful representation for nilpotent Lie algebras of dimension at most $5$}\label{Sec:22}


The brackets, for each Lie algebra
$\g$ considered bellow, are described in a basis
$B = \{X_1, \dots, X_p, Z_1, \dots, Z_q\}$ such that
$\{Z_1, \dots, Z_q\}$ is a basis of the center
$\mathfrak{z}(\g)$ of the Lie algebra
$\g$.

The Table \ref{Tabla:3} contains the classification of all nilpotent Lie algebras of dimension at most
$5$.

\begin{table}[!h]
\begin{centering}
$$
\begin{tabular}{||c|c|c||}
\hline
\multicolumn{1}{||c||}{{\bf Dim.}} & {\bf Lie algebra} & {\bf Law of the Lie algebra}\\
\hline
\hline
\vbox to .4cm{}
1                  & $L_{1,1}$ & abelian \\
\hline
\vbox to .4cm{}
2                  & $L_{2,1}$  & abelian \\
\hline
\vbox to .4cm{}
\multirow{2}{*}{3} & $L_{3,1}$ & abelian  \\
\cdashline{2-3}
                   & $L_{3,2}$ \vbox to .4cm{}& $[X_1 , X_2]= Z_1$ \\
\hline
\multirow{3}{*}{4} & $L_{4,1}$ \vbox to .4cm{}& abelian \\
\cdashline{2-3}
                   & $L_{4,2}$ \vbox to .4cm{} &  $L_{3,2} \oplus L_{1,1}$ \\
\cdashline{2-3}
                   & $L_{4,3}$  \vbox to .4cm{}& $[X_1, X_2]= X_3, [X_1, X_3]= Z_1$ \\
\hline
\vbox to .4cm{}
\multirow{10}{*}{5} & $L_{5,1}$ \vbox to .4cm{}& abelian \\
\cdashline{2-3}
                   & $L_{5,2}$ \vbox to .4cm{}& $L_{3,2} \oplus L_{2,1}$ \\
\cdashline{2-3}
                   & $L_{5,3}$ \vbox to .4cm{}& $L_{4,3} \oplus L_{1,1}$ \\
\cdashline{2-3}
                   & $L_{5,4}$ \vbox to .4cm{} & $[X_1, X_2]= Z_1, [X_3, X_4]= Z_1$ \\
\cdashline{2-3}
                  & $L_{5,5}$ \vbox to .4cm{} & $[X_1, X_2] = X_3, [X_1, X_3] = Z_1, [X_2, X_4] = Z_1$ \\
\cdashline{2-3}
                  & \multirow{2}{*}{$L_{5,6}$}   & \vbox to .4cm{}$[X_1, X_2] = X_3, [X_1, X_3] = X_4, [X_1, X_4] = Z_1$,\\
                  &                                     & $[X_2, X_3]= Z_1$ \\
\cdashline{2-3}
                 & $L_{5,7}$  \vbox to .4cm{}              & $[X_1, X_2] = X_3, [X_1, X_3] = X_4, [X_1, X_4] = Z_1$ \\
\cdashline{2-3}
                        & $L_{5,8}$ \vbox to .4cm{}               & $[X_1, X_2] = Z_1, [X_1, X_3] = Z_1$ \\
\cdashline{2-3}
                        & $L_{5,9}$  \vbox to .4cm{}              & $[X_1, X_2] = X_3, [X_1, X_3] = Z_1, [X_2, X_3] = Z_2$ \\
\hline
\end{tabular}
$$
\caption{Lie algebras of $\dim \leq 5$.}
\label{Tabla:3}
\end{centering}
\end{table}

The Tables \ref{Tabla:4} and \ref{Tabla:6} contain a faithful minimal representation and a faithful minimal nilrepresentation of each Lie algebra of
dimension $\leq 5$ and in the last columns contains the necessary reference for the lower bound of
$\mu \text{ and } \mu_{nil}$. The matrices appearing in the tables stand for
$$\pi\left(\sum x_i X_i + \sum z_i Z_i + \sum a_i A_i\right),$$
where
$\{Z_1, \dots, Z_q, A_1, \dots, A_r\}$ is a basis of
$\mathfrak{z}(\g)$ such that
$\{Z_1, \dots, Z_q\}$ is a basis of
$\mathfrak{z}(\g)\cap [\g, \g]$ and
$\{A_1, \dots, A_r\}$ is a basis of a complement of any subspace of
$\mathfrak{z}(\g)\cap [\g, \g]$ in
$\mathfrak{z}(\g)$.
For instance, let
$\{X_1, X_2, Z_1, A_1\}$ be a basis of
$L_{4,2}$ and let
$\pi: L_{4,2} \rightarrow \mathfrak{gl}(\k^3)$ be a faithful representation of
$L_{4,2}$ defined by
$$
\pi\left(x_1 X_1 + x_2 X_2 + z_1 Z_1 + a_1 A_1 \right)= \left[
                                                        \begin{array}{ccc}
                                                        a_1 & x_1 & z_1 \\
                                                        0   & a_1 & x_2 \\
                                                        0   &  0  &  a_1
                                                        \end{array}
                                                        \right].
$$

\begin{table}[!h]
\begin{centering}
$$
\begin{tabular}{||c|c|c|c|c||c||}
\hline
$\pmb{\g}$ & $\pmb{\mu}_{nil}$ & $\pmb{\mu}$ & {\bf faith. nilrep.} & {\bf faith. rep.}& {\bf ref.}\\
\hline
\hline
\vbox to .5cm{}
$L_{1,1}$ & 2 & 1 & $\left[
                                           \begin{smallmatrix}
                                           0 & a_1 \\
                                           0 & 0
                                           \end{smallmatrix}
                                           \right]$                                         & $\left[
                     \begin{smallmatrix}
                     a_1
                     \end{smallmatrix}
                     \right]$    & Prop.\ref{munilabeliana}(\ref{abelian})\\
\hline
\vbox to .6cm{}
$L_{2,1}$ & 3 & 2  & $\left[
                                                   \begin{smallmatrix}
                                                   0 & a_1 & a_2 \\
                                                   0 & 0 & 0 \\
                                                   0 & 0 & 0
                                                   \end{smallmatrix}
                                                   \right]$
                                                & $\left[
                        \begin{smallmatrix}
                        a_1 & a_2 \\
                        0   & a_1
                        \end{smallmatrix}
                        \right]$   & Prop.\ref{munilabeliana}(\ref{abelian}) \\
\hline
\vbox to .8cm{}
$L_{3,1}$ & 4 & 3 & $\left[
                                                \begin{smallmatrix}
                                                0 & 0 & a_1 & a_2 \\
                                                0 & 0 & 0 & a_3 \\
                                                0 & 0 & 0 & 0 \\
                                                0 & 0 & 0 & 0
                                                \end{smallmatrix}
                                                \right]$    & $\left[
                       \begin{smallmatrix}
                       a_1 & 0 & a_2 \\
                       0 & a_1 & a_3 \\
                       0 & 0 & a_1
                       \end{smallmatrix}
                       \right]$  & Prop.\ref{munilabeliana}(\ref{abelian})\\
\hline
\vbox to .8cm{}
$L_{4,2}$  & 4 & 3 & $\left[
                                              \begin{smallmatrix}
                                              0 & x_1 & z & a_1 \\
                                              0 & 0 & x_2 & 0\\
                                              0 & 0 & 0 & 0\\
                                              0 & 0 & 0 & 0
                                              \end{smallmatrix}
                                              \right]$  & $\left[
                        \begin{smallmatrix}
                        a_1 & x_1 & z_1 \\
                        0 & a_1 & x_2 \\
                        0 & 0 & a_1
                        \end{smallmatrix}
                        \right]$   & \begin{tabular}{l}
                                      \cite{R}, since $L_{4,2}$\\
                                      is $L_{3,2} \oplus L_{1,1}$\\
                                      \end{tabular}\\
\hline
\vbox to .9cm{}
$L_{5,1}$ & 5 & 4 & $\left[
                                                   \begin{smallmatrix}
                                                   0 & 0 & a_1 & a_2 & a_3 \\
                                                   0 & 0 & 0 & a_4 & a_5 \\
                                                   0 & 0 & 0 & 0 & 0 \\
                                                   0 & 0 & 0 & 0 & 0 \\
                                                   0 & 0 & 0 & 0 & 0
                                                   \end{smallmatrix}
                                                   \right]$  & $\left[
                         \begin{smallmatrix}
                         a_1 & 0 & a_2 & a_3 \\
                         0 & a_1 & a_4 & a_5 \\
                         0 & 0 & a_1 & 0 \\
                         0 & 0 & 0 & a_1
                         \end{smallmatrix}
                         \right]$ & Prop.\ref{munilabeliana}(\ref{abelian})\\
\hline
\vbox to .9cm{}
$L_{5,2}$ & 5 & 4 & $\left[
                                                \begin{smallmatrix}
                                                0 & x_1 & z_1 & a_1 & a_2 \\
                                                0 & 0   & x_2 & 0 & 0 \\
                                                0 & 0   & 0   & 0 & 0 \\
                                                0 & 0   & 0   & 0 & 0 \\
                                                0 & 0   & 0   & 0 & 0
                                                \end{smallmatrix}
                                                \right]$   &   $\left[
                        \begin{smallmatrix}
                        a_1 & x_1 & z_1 & 0 \\
                        0 & a_1 & x_2 & 0 \\
                        0 & 0 & a_1 & 0 \\
                        0 & 0 & 0 & a_2
                        \end{smallmatrix}
                        \right]$    & \begin{tabular}{l}
                                      \cite{R}, since $L_{5,2}$\\
                                      is $L_{3,2} \oplus L_{2,1}$\\
                                      \end{tabular}\\
\hline
\end{tabular}
$$
\caption{Lie algebras such that $\mu < \mu_{nil}$.}
\label{Tabla:4}
\end{centering}
\end{table}

In the Table \ref{Tabla:6} all the Lie algebras, except
$L_{4,1} \text{ and } L_{5,3}$, verify that
$\mathfrak{z}(\g) \subseteq [\g, \g]$. Then, by Corollary \ref{coro:mumunil} we have
$\mu_{nil}(\g) = \mu(\g)$.

$$
\begin{tabular}{||c|c|c||c||}
\hline
$\pmb{\g}$ & $\pmb{\mu}$ & {\bf faithful nilrepresentation}  & {\bf ref.}\\
\hline
\hline
\vbox to .6cm{}
$L_{3,2}$ & 3 & $\left[\begin{smallmatrix}
             0 & x_1 & z_1 \\
             0 & 0 & x_2 \\
             0 & 0 & 0
             \end{smallmatrix}
             \right]$ & \begin{tabular}{l}
                        by Engel's Theorem, \\
                       we have $\mu_{nil}(L_{3,2}) \geq 3$\\
                       \end{tabular}\\
\hline
$L_{4,1}$ & 4  & \vbox to .8cm{}$\left[
              \begin{smallmatrix}
              0 & 0 & a_1 & a_2 \\
              0 & 0 & a_3 & a_4 \\
              0 & 0 & 0 & 0 \\
              0 & 0 & 0 & 0
              \end{smallmatrix}
              \right]$  & Prop.\ref{munilabeliana}(\ref{abelian}) \\
\hline
$L_{4,3}$ & 4 & \vbox to .9cm{}$\left[
                \begin{smallmatrix}
                0 & x_1 & 0 & z_1 \\
                0 & 0 & x_1 & x_3 \\
                0 & 0 & 0   & x_2 \\
                0 & 0 & 0   & 0
                \end{smallmatrix}
                \right]$    & Prop.\ref{munilabeliana}(\ref{filiform})\\
\hline
\end{tabular}
$$

\begin{table}[!h]
\begin{centering}
$$
\begin{tabular}{||c|c|c||c||}
\hline
$\pmb{\g}$ & $\pmb{\mu}$ & {\bf faithful nilrepresentation}  & {\bf ref.}\\
\hline
\hline
\vbox to .8cm{}
$L_{5,3}$ & 4 & $\left[
                \begin{smallmatrix}
                0 & x_1 & x_3 + a_1 & -2z_1 \\
                0 & 0 & x_2 & -x_3 + a_1 \\
                0 & 0 & 0 & x_1 \\
                0 & 0 & 0 & 0
                \end{smallmatrix}
                \right]$ &  Prop.\ref{munilabeliana}(\ref{directsum})\\
\hline
\vbox to .8cm{}
$L_{5,4}$ & 4 & $\left[
                 \begin{smallmatrix}
                         0 & x_1 & x_3 & z_1 \\
                         0 & 0 & 0 & x_2 \\
                         0 & 0 & 0 & x_4\\
                         0 & 0 & 0 & 0
                         \end{smallmatrix}
                         \right]$ & \begin{tabular}{l}
                                     by Engel's Theorem, \\
                                     we have $\mu_{nil}(L_{5,4}) \geq 4$ \\
                                     \end{tabular} \\
\hline
\vbox to .8cm{}
$L_{5,5}$ & 4 & $\left[
                                                \begin{smallmatrix}
                                                0 & x_1 & -x_4 & z_1 \\
                                                0 & 0 & x_1 & x_3 \\
                                                0 & 0 & 0 & x_2 \\
                                                0 & 0 & 0 & 0
                                                \end{smallmatrix}
                                                \right]$ & \begin{tabular}{l}
                                     by Engel's Theorem, \\
                                     we have $\mu_{nil}(L_{5,5}) \geq 4$ \\
                                     \end{tabular} \\
\hline
\vbox to 1cm{}
$L_{5,6}$ & 5 & $\left[
                        \begin{smallmatrix}
                        0 & x_1 & \frac{1}{2}x_2 & -\frac{1}{2}x_3 & z_1 \\
                        0 & 0 & x_1            & 0               & x_4 \\
                        0 & 0 & 0              & x_1             & x_3 \\
                        0 & 0 & 0              & 0               & x_2 \\
                        0 & 0 & 0              & 0               & 0
                        \end{smallmatrix}
                        \right]$ & Prop.\ref{munilabeliana}(\ref{filiform})\\
\hline
\vbox to 1cm{}
$L_{5,7}$ & 5 & $\left[
                                       \begin{smallmatrix}
                                        0 & x_1 & 0     &0     & z_1 \\
                                        0 & 0   & x_1   & 0    & x_4 \\
                                        0 & 0   & 0     & x_1  & x_3 \\
                                        0 & 0   & 0     & 0    & x_2 \\
                                        0 & 0   & 0     & 0    & 0
                                       \end{smallmatrix}
                                       \right]$ & Prop.\ref{munilabeliana}(\ref{filiform})\\
\hline
\vbox to .8cm{}
$L_{5,8}$ & 4 & $\left[
                         \begin{smallmatrix}
                         0 & x_1 & z_1 & z_2 \\
                         0 & 0 & x_2 & x_3 \\
                         0 & 0 & 0 & 0 \\
                         0 & 0 & 0 & 0
                         \end{smallmatrix}
                         \right]$ & \begin{tabular}{l}
                                     by Engel's Theorem, \\
                                     we have $\mu_{nil}(L_{5,8}) \geq 4$ \\
                                     \end{tabular} \\
\hline
\vbox to 1cm{}
$L_{5,9}$ & 5 & $\left[
                                       \begin{smallmatrix}
                                        0 & 0 & \frac{1}{2}x_2 & -\frac{1}{2}x_3 & z_2 \\
                                        0 & 0 & x_1            & 0               & z_1 \\
                                        0 & 0 & 0 & x_1 & x_3 \\
                                        0 & 0 & 0 & 0 & x_2 \\
                                        0 & 0 & 0 & 0 & 0
                                       \end{smallmatrix}
                                       \right]$  & Lemma \ref{lema:59} \\
\hline
\end{tabular}
$$
\caption{Lie algebras such that $\mu = \mu_{nil}$.}
\label{Tabla:6}
\end{centering}
\end{table}

\begin{lemma}\label{lema:59}
Let
$n \in \mathbb{N}$ and let
$\g$ be a Lie subalgebra of
$\mathfrak{n}_n(\k)$ isomorphic to
$L_{5,9}$. Then
$n \geq 5$.
\end{lemma}

\begin{proof}
Suppose, contrary to our claim, that
$n= 4$. Since
$\g$ is isomorphic to
$L_{5,9}$, we get
$\dim [\g, [\g, \g]]= 2$ and since
$\g$ is a Lie subalgebra of
$\mathfrak{n}_4(\k)$, we have
$[\g, [\g, \g]] \subseteq [\mathfrak{n}_4(\k), [\mathfrak{n}_4(\k), \mathfrak{n}_4(\k)]]$. It follows that
$$2 \leq \dim [\mathfrak{n}_4(\k), [\mathfrak{n}_4(\k), \mathfrak{n}_4(\k)]],$$
which is impossible and the proof is complete.
\end{proof}

\end{section}


\section{Main Theorem}



\subsection{Nilpotent Lie algebras of dimension $6$}\label{sec:dimen6}


In dimension
$6$ we used the De Graaf's classification (see \cite{Gr}). The brackets, for each Lie algebra
$L_{6,j}$ considered bellow, are described in a basis
$B = \{X_1, \dots, X_p, Z_1, \dots, Z_q\}$ such that
$\{Z_1, \dots, Z_q\}$ is a basis of the center
$\mathfrak{z}(L_{6,j})$ of the Lie algebra
$L_{6,j}$.

\noindent \begin{tabular}{clcl}
${\bullet}$ & $L_{6,j}$ & $=$ & $L_{5,j} \oplus \k$ for all $j= 1, \dots, 9$.\\
${\bullet}$ & $L_{6,10}$ & $:$ & $[X_1, X_2] = X_3, [X_1, X_3] = Z_1, [X_4, X_5] = Z_1$.\\
${\bullet}$ & $L_{6,11}$ & $:$ & $[X_1, X_2] = X_3, [X_1, X_3] = X_4, [X_1, X_4] = Z_1, [X_2, X_3] = Z_1$,\\
            &            &     & $[X_2, X_5] = Z_1$.\\
${\bullet}$ & $L_{6,12}$ & $:$ & $[X_1, X_2] = X_3, [X_1, X_3] = X_4, [X_1, X_4] = Z_1, [X_2, X_5] = Z_1$.\\
\end{tabular}

\noindent \begin{tabular}{clcl}
${\bullet}$ & $L_{6,13}$ & $:$ & $[X_1, X_2] = X_3, [X_1, X_3] = X_5, [X_2, X_4] = X_5, [X_1, X_5] = Z_1$,\\
            &            &     & $[X_3 , X_4] = Z_1$.\\
${\bullet}$ & $L_{6,14}$ & $:$ & $[X_1, X_2] = X_3, [X_1, X_3] = X_4, [X_1, X_4] = X_5, [X_2, X_3] = X_5$,\\
            &            &     & $[X_2 , X_5] = Z_1, [X_3, X_4] = - Z_1$.\\
${\bullet}$ & $L_{6,15}$ & $:$ & $[X_1, X_2] = X_3, [X_1, X_3] = X_4, [X_1, X_4] = X_5, [X_2, X_3] = X_5$,\\
            &            &     & $ [X_1 , X_5] = Z_1, [X_2, X_4] = Z_1$.\\
${\bullet}$ & $L_{6,16}$ & $:$ & $[X_1, X_2] = X_3, [X_1, X_3] = X_4,  [X_1, X_4] = X_5, [X_2, X_5] = Z_1$,\\
            &            &     & $[X_3 , X_4] = -Z_1$.\\
${\bullet}$ & $L_{6,17}$ & $:$ & $[X_1, X_2] = X_3, [X_1, X_3] = X_4,  [X_1, X_4] = X_5, [X_1, X_5] = Z_1$,\\
            &            &     & $[X_2 , X_3] = Z_1$.\\
${\bullet}$ & $L_{6,18}$ & $:$ & $[X_1, X_2] = X_3, [X_1, X_3] = X_4,  [X_1, X_4] = X_5, [X_1, X_5] = Z_1$.\\
${\bullet}$ & $L_{6,19}(\epsilon)$ & $:$ & $[X_1, X_2] = X_4 , [X_1, X_3] = X_5  ,  [X_2, X_4] = Z_1, [X_3, X_5] = \epsilon Z_1$\\
            &                      &     & Isomorphism: $L_{6,19}(\epsilon) \cong L_{6,19}(\delta)$ if and only if there is and \\
            &                      &     & $\alpha \in \k^*$ such that $\delta = \alpha^2 \epsilon$.\\
${\bullet}$ & $L_{6,20}$ & $:$ & $[X_1, X_2] = X_4 , [X_1, X_3] = X_5  ,  [X_1, X_5] = Z_1, [X_2, X_4] = Z_1$.\\
${\bullet}$ & $L_{6,21}(\epsilon)$ & $:$ & $[X_1, X_2] = X_3, [X_1, X_3] = X_4  ,  [X_2, X_3] = X_5 , [X_1, X_4] = Z_1$,\\
            &                      &     & $ [X_2, X_5] = \epsilon Z_1$. Isomorphism: $L_{6,21}(\epsilon) \cong L_{6,21}(\delta)$ if and only\\
            &                      &     & if there is  and $\alpha \in \k^*$ such that $\delta = \alpha^2 \epsilon$.\\
${\bullet}$ & $L_{6,22}(\epsilon)$ & $:$ & $[X_1, X_2] = Z_1, [X_1, X_3] = Z_2,  [X_2, X_4] = \epsilon Z_2, [X_3, X_4] = Z_1$\\
            &                      &     & Isomorphism: $L_{6,22}(\epsilon) \cong L_{6,22}(\delta)$ if and only if there is and\\
            &                      &     & $\alpha \in \k^*$ such that $\delta = \alpha^2 \epsilon$.\\
${\bullet}$ & $L_{6,23}$ & $:$ & $[X_1, X_2] = X_3 , [X_1, X_3] = Z_1,  [X_1, X_4] = Z_2, [X_2, X_4] = Z_1$.\\
${\bullet}$ & $L_{6,24}(\epsilon)$ & $:$ & $[X_1, X_2] = X_3 , [X_1, X_3] = Z_1,  [X_1, X_4] = \epsilon Z_2, [X_2, X_3] = Z_2$,\\
            &                      &     & $ [X_2, X_4] = Z_1$ Isomorphism: $L_{6,24}(\epsilon) \cong L_{6,24}(\delta)$ if and only\\
            &                      &     & if there is and $\alpha \in \k^*$ such that $\delta = \alpha^2 \epsilon$.\\
${\bullet}$ & $L_{6,25}$ & $:$ &  $[X_1, X_2] = X_3 , [X_1, X_3] = Z_1,  [X_1, X_4] = Z_2$.\\
${\bullet}$ & $L_{6,26}$ & $:$ &  $[X_1, X_2] = Z_1, [X_1, X_3] = Z_2,  [X_2, X_3] = Z_3$.\\

\end{tabular}


\subsection{Minimal faithful representation and faithful nilrepresentation for the nilpotent Lie algebras of dimension $6$}\label{sec:dim6}


Let
$\mathfrak{t}_n(\k)$ be the set of upper triangular matrices of size
$n$ over a field
$\k$.
Let
$\k$ be an algebraically closed field and let
$(\pi, V)$ be a faithful representation of
$L_{6,j}$. If
$\dim V= n$, by Lie's Theorem there exist a basis
$B$ of
$V$ such that
$[\pi(X)]_B \in \mathfrak{t}_n(\k)$ for all
$X \in L_{6,j}$. Since
$\dim L_{6,j} = 6$ and
$L_{6,j}$ is a nilpotent Lie algebra, we obtain
\begin{equation}\label{eq:Lie}
4 \leq \mu(L_{6,j}) \;\;\;\text{ for all } j= 1, \dots, 26.
\end{equation}

Now assume that
$\k$ is arbitrary (with
$\text{char}(\k)=0$) and let
$\g$ be a Lie algebra over
$\k$. Let
$\overline{\k}$ be the algebraic closure of
$\k$. The tensor product
$\g \otimes \overline{\k}$ has a Lie algebra structure with bracket
$$
[X_1 \otimes a_1, X_2 \otimes a_2] = [X_1, X_2] \otimes a_1a_2,
$$
$X_i \in \g, a_i \in \overline{\k}$. Note that
$\g \otimes \overline{\k}$ could viewed as a Lie algebra over
$\overline{\k}$.

Let
$(\pi, V)$ be a representation of
$\g$. We define the linear mapping
$\widetilde{\pi}: \g \otimes \overline{\k} \rightarrow \mathfrak{gl}(V \otimes \overline{\k})$ over
$\overline{\k}$ given by
$\widetilde{\pi}(X \otimes a)(v \otimes b) = \pi(X)(v) \otimes ba$. By a straightforward calculation we have
$(\widetilde{\pi}, V \otimes \overline{\k})$ is a representation of
$\g \otimes \overline{\k}$ over
$\overline{\k}$ and if
$(\pi, V)$ is faithful then
$(\widetilde{\pi}, V \otimes \overline{\k})$ is faithful too. Thus
\begin{equation}\label{eq:tensor}
\mu(\g \otimes \overline{\k}) \leq \mu(\g)
\end{equation}

Therefore, if
$\g$ is a nilpotent Lie algebra over a field
$\k$ of
$\text{char}(\k) = 0$ by equations (\ref{eq:Lie}), (\ref{eq:tensor}) and since
$\mu(\g) \leq \mu_{nil}(\g)$ and  we get
\begin{equation}\label{eq:cotainf}
4 \leq \mu(\g) \;\; \text{ and } \;\; 4 \leq \mu_{nil}(\g).
\end{equation}

The Tables \ref{Tabla:7} and \ref{Tabla:8} contain a faithful minimal representation, and a faithful minimal nilrepresentation, of each the Lie algebra
$L_{6,j}$ for all
$j= 1, \dots,26$ and in the last columns contains the necessary reference for the lower bound of
$\mu \text{ and } \mu_{nil}$.
The matrices appearing in the tables stand for
$$\pi\left(\sum x_i X_i + \sum z_i Z_i + \sum a_i A_i\right),$$
where
$\{Z_1, \dots, Z_q, A_1, \dots, A_r\}$ is a basis of
$\mathfrak{z}(\g)$ such that
$\{Z_1, \dots, Z_q\}$ is a basis of
$\mathfrak{z}(\g)\cap [\g, \g]$ and
$\{A_1, \dots, A_r\}$ is a basis of a complement any subspace of
$\mathfrak{z}(\g)\cap [\g, \g]$ in
$\mathfrak{z}(\g)$.

\begin{table}[!h]
\begin{centering}
$$
\begin{tabular}{||c|c|c|c|c||c||}
\hline
$\pmb{\g}$ & $\pmb{\mu}_{n}$ & $\pmb{\mu}$ &  {\bf faithful nilrepresent.} & {\bf faithful represent.} & {\bf ref.}\\
\hline
\hline
\vbox to .8cm{}
$L_{6,3}$ & 5 & 4 & $\left[
                    \begin{smallmatrix}
                    0   & x_1 & x_3 + a_1   &   -2z_1  & a_2\\
                    0   & 0   & x_2         & -x_3 + a_1 & 0 \\
                    0   &  0  & 0           & x_1 & 0 \\
                    0   &  0  & 0           & 0   & 0 \\
                    0   &  0  & 0           & 0   & 0
                    \end{smallmatrix}
                    \right]$ &
                               $\left[
                               \begin{smallmatrix}
                               a_2 & x_1 & x_3 + a_1   &   -2z_1 \\
                               0   & a_2 & x_2         & -x_3 + a_1 \\
                               0   &  0  & a_2         & x_1 \\
                               0   &  0  & 0           & a_2
                               \end{smallmatrix}
                               \right]$ & \begin{tabular}{l}
                                           - Eq.(\ref{remark:L19})\\
                                           - Eq.(\ref{eq:cotainf})\\
                                           \end{tabular}\\
\hline
\vbox to .9cm{}
$L_{6,4}$ & 5 & 4 & $\left[
                     \begin{smallmatrix}
                     0 & x_1 & x_2 & z_1  & a_1\\
                     0 & 0   & 0 & x_3  & 0 \\
                     0 & 0 & 0   & x_4  & 0 \\
                     0 & 0 & 0 & 0  & 0 \\
                     0 & 0 & 0 & 0  & 0
                     \end{smallmatrix}
                     \right]$       &    $\left[
                                          \begin{smallmatrix}
                                          a_1 & x_1 & x_2 & z_1 \\
                                          0 & a_1 & 0 & x_3 \\
                                          0 & 0 & a_1 & x_4 \\
                                          0 & 0 & 0 & a_1
                                          \end{smallmatrix}
                                          \right]$ & \begin{tabular}{l}
                                           -Eq.(\ref{remark:L19})\\
                                           -Eq.(\ref{eq:cotainf})\\
                                           \end{tabular}\\
\hline
\vbox to .9cm{}
$L_{6,5}$ & 5 & 4 & $\left[
\begin{smallmatrix}
0 & x_1 & -x_4 & z_1  & a_1 \\
0 & 0   & x_1 & x_3 & 0 \\
0 & 0   & 0   & x_2 & 0\\
0 & 0   & 0   & 0   & 0\\
0 & 0   & 0   & 0   & 0
\end{smallmatrix}
\right]$ & $\left[
\begin{smallmatrix}
a_1 & x_1 & -x_4 & z_1 \\
0 & a_1 & x_1 & x_3 \\
0 & 0 & a_1 & x_2 \\
0 & 0 & 0 & a_1
\end{smallmatrix}
\right]$ & \begin{tabular}{l}
           -Eq.(\ref{remark:L19})\\
           -Eq.(\ref{eq:cotainf})\\
           \end{tabular}\\
\hline
\vbox to .9cm{}
$L_{6,8}$ & 5 & 4 & $\left[
\begin{smallmatrix}
0 & x_1 & z_1  & z_2 & a_1\\
0   & 0 & x_2  & x_3 & 0\\
0   & 0   & 0  & 0  & 0\\
0   & 0   & 0  & 0 & 0 \\
0   & 0   & 0  & 0 & 0
\end{smallmatrix}
\right]$ &  $\left[
\begin{smallmatrix}
a_1 & x_1 & z_1  & z_2 \\
0   & a_1 & x_2  & x_3\\
0   & 0   & a_1  & 0 \\
0   & 0   & 0    & a_1
\end{smallmatrix}
\right]$ & \begin{tabular}{l}
           -Eq.(\ref{remark:L19})\\
           -Eq.(\ref{eq:cotainf})\\
           \end{tabular}\\
\hline
\vbox to 1.05cm{}
$L_{6,9}$ & 6 & 5 & $\left[
\begin{smallmatrix}
0   & 0   & \frac{1}{2}x_2 & -\frac{1}{2}x_3 & z_2  & a_1\\
0   & 0   & x_1            & 0               & z_1  & 0\\
0   & 0   & 0              & x_1 & x_3 & 0\\
0   & 0   & 0   & 0        & x_2 & 0\\
0 & 0 & 0 & 0 & 0   & 0 \\
0 & 0 & 0 & 0 & 0   & 0
\end{smallmatrix}
\right]$ & $\left[
\begin{smallmatrix}
a_1 & 0   & \frac{1}{2}x_2 & -\frac{1}{2}x_3 & z_2 \\
0   & a_1 & x_1            & 0               & z_1 \\
0   & 0   & a_1 & x_1 & x_3 \\
0   & 0   & 0   & a_1 & x_2 \\
0 & 0 & 0 & 0 & a_1
\end{smallmatrix}
\right]$ &  \begin{tabular}{l}
           - Cor.\ref{coro:L69}\\
           - Pr.\ref{munilabeliana}(\ref{directsum})\\
           \end{tabular}\\
\hline
\end{tabular}
$$
\caption{Lie algebras such that $\mu < \mu_{nil}$.}
\label{Tabla:7}
\end{centering}
\end{table}

In the Table \ref{Tabla:8} all the Lie algebras, except the Lie algebras
$L_{6,1}, L_{6,2}$ and
$L_{6,7}$, verify that
$\mathfrak{z}(\g) \subseteq [\g, \g]$ then, by Corollary \ref{coro:mumunil} we have
$\mu_{nil}(\g) = \mu(\g)$.

$$
\begin{tabular}{||c|c|c||c||}
\hline
$\pmb{\g}$ & $\pmb{\mu}$ &  {\bf faithful nilrepresent.} & {\bf ref.}\\
\hline
\vbox to .9cm{}
$L_{6,1}$ & 5 & $\left[
                        \begin{smallmatrix}
                        0 & 0 & a_1 & a_2 & a_3 \\
                        0 & 0 & a_4 & a_5 & a_6 \\
                        0 & 0 & 0   & 0   & 0 \\
                        0 & 0 & 0   & 0   & 0 \\
                        0 & 0 & 0   & 0   & 0
                        \end{smallmatrix}
                        \right]$ & Pr.\ref{munilabeliana}(\ref{abelian})\\
\hline
\vbox to .9cm{}
$L_{6,2}$ & 5 & $\left[
                    \begin{smallmatrix}
                                                    0 & 0 & x_1 & z_1 & a_1 \\
                                                    0 & 0 & 0   & a_2 & a_3 \\
                                                    0 & 0 & 0   & x_2   & 0 \\
                                                    0 & 0 & 0   & 0   & 0 \\
                                                    0 & 0 & 0   & 0   & 0
                                                    \end{smallmatrix}
                                                    \right]$ & \cite{R}\\
\hline
\end{tabular}
$$

$$
\begin{tabular}{||c|c|c||c|}
\hline
$\pmb{\g}$ & $\pmb{\mu}$ &  {\bf faithful nilrepresent.} & {\bf ref.}\\
\hline
\hline
\vbox to 1.cm{}
$L_{6,6}$ & 5 & $\left[\begin{smallmatrix}
                 0   & x_1 & 3x_2   & x_4 + a_1  & -3z_1 \\
                 0   & 0   & x_1 & x_3   & -2x_4 + a_1\\
                 0   & 0   & 0   & x_2   & -x_3 \\
                 0   & 0   & 0   & 0     & x_1 \\
                 0   & 0   & 0   & 0     & 0
                 \end{smallmatrix}
                 \right]$ & Pr.\ref{munilabeliana}(\ref{directsum})\\
\hline
\vbox to 1.05cm{}
$L_{6,7}$ & 5 & $\left[\begin{smallmatrix}
                 0   & x_1 & 0   & x_4 + a_1  & -3z_1 \\
                 0   & 0   & x_1 & x_3   & -2x_4 + a_1\\
                 0   & 0   & 0   & x_2   & -x_3 \\
                 0   & 0   & 0   & 0     & x_1 \\
                 0   & 0   & 0   & 0     & 0
                 \end{smallmatrix}
                 \right]$ & Pr.\ref{munilabeliana}(\ref{directsum})\\
\hline
\vbox to 1.05cm{}
$L_{6,10}$ & 5 & $\left[
\begin{smallmatrix}
0   & x_1 & 0      & x_4 & z_1 \\
0   & 0   & x_1    & 0   & x_3 \\
0   & 0   & 0      & 0   & x_2 \\
0   & 0   & 0      & 0   & x_5 \\
0   & 0   & 0      & 0   & 0
\end{smallmatrix}
\right]$ & Eq.(\ref{remark:L19})\\
\hline
\vbox to 1.05cm{}
$L_{6,11}$ & 5 &$\left[
\begin{smallmatrix}
0 & x_1 & x_2 & -x_5 & z_1 \\
0 & 0 & x_1 & x_2 & x_4 \\
0 & 0 & 0 & x_1 & x_3 \\
0 & 0 & 0 & 0 & x_2 \\
0 & 0 & 0 & 0 & 0
\end{smallmatrix}
\right]$ & Eq.(\ref{remark:L19})\\
\hline
\vbox to 1.05cm{}
$L_{6,12}$ & 5 & $\left[
\begin{smallmatrix}
0 & x_1 & 0 & -x_5 & z_1 \\
0 & 0 & x_1 & 0 & x_4 \\
0 & 0 & 0 & x_1 & x_3 \\
0 & 0 & 0 & 0 & x_2 \\
0 & 0 & 0 & 0 & 0
\end{smallmatrix}
\right]$ & Eq.(\ref{remark:L19})\\
\hline
\vbox to 1.05cm{}
$L_{6,13}$ & 5 & $\left[
\begin{smallmatrix}
0 & x_1 & -x_4 & 0 & z_1\\
0 & 0 & x_1 & -x_4 & x_5 \\
0 & 0 & 0 & x_1 & x_3 \\
0 & 0 & 0 & 0 & x_2 \\
0 & 0 & 0 & 0 & 0
\end{smallmatrix}
\right]$  & Eq.(\ref{remark:L19})\\
\hline
\vbox to 1.1cm{}
$L_{6,14}$ & 6 & $\left[
\begin{smallmatrix}
0 & x_2 & -x_3 & 0              & 0                & z_1 \\
0 &  0  &  x_1 & \frac{1}{2}x_2 & - \frac{1}{2}x_3 & x_5 \\
0 &  0  &  0   &     x_1        &     0            & x_4 \\
0 &  0  &  0   &     0          &     x_1          & x_3 \\
0 &  0  &  0   &     0          &     0            & x_2 \\
0 &  0  &  0   &     0          &     0            & 0
\end{smallmatrix}
\right]$ & Pr.\ref{munilabeliana}(\ref{filiform})\\
\hline
\vbox to 1.15cm{}
$L_{6,15}$ & 6 & $\left[
\begin{smallmatrix}
0 & x_1 & \frac{1}{2}x_2 &    0            & -\frac{1}{2}x_4 & z_1 \\
0 &  0  &  x_1           & \frac{1}{2}x_2  & - \frac{1}{2}x_3 & x_5 \\
0 &  0  &  0             &     x_1         &     0            & x_4 \\
0 &  0  &  0             &     0           &     x_1          & x_3 \\
0 &  0  &  0             &     0           &     0            & x_2 \\
0 &  0  &  0             &     0           &     0            & 0
\end{smallmatrix}
\right]$ & Pr.\ref{munilabeliana}(\ref{filiform})\\
\hline
\vbox to 1.1cm{}
$L_{6,16}$ & 6 & $\left[
\begin{smallmatrix}
0 & x_1 & x_3 & -2x_4  & 3x_5   & -3z_1 \\
0 &  0  &  x_2 & -x_3  & x_4   & 0 \\
0 &  0  &  0   &     x_1        &    0             & x_4 \\
0 &  0  &  0   &     0          &     x_1          & x_3 \\
0 &  0  &  0   &     0          &     0            & x_2 \\
0 &  0  &  0   &     0          &     0            & 0
\end{smallmatrix}
\right]$ & Pr.\ref{munilabeliana}(\ref{filiform})\\
\hline
\vbox to 1.1cm{}
$L_{6,17}$ & 6 & $\left[
\begin{smallmatrix}
0 & x_1 & 0 & \frac{1}{2}x_2    & -\frac{1}{2}x_3  & z_1 \\
0 &  0  &  x_1 & 0    & 0   & x_5 \\
0 &  0  &  0   & x_1  & 0   & x_4 \\
0 &  0  &  0   &  0   & x_1 & x_3 \\
0 &  0  &  0   &  0   &  0  & x_2 \\
0 &  0  &  0   &  0   &  0  & 0
\end{smallmatrix}
\right]$  & Pr.\ref{munilabeliana}(\ref{filiform})\\
\hline
\vbox to 1.1cm{}
$L_{6,18}$ & 6 & $\left[
\begin{smallmatrix}
0 & x_1 & 0 & 0    & 0   & z_1 \\
0 &  0  &  x_1 & 0    & 0   & x_5 \\
0 &  0  &  0   & x_1  & 0   & x_4 \\
0 &  0  &  0   &  0   & x_1 & x_3 \\
0 &  0  &  0   &  0   &  0  & x_2 \\
0 &  0  &  0   &  0   &  0  & 0
\end{smallmatrix}
\right]$ & Pr.\ref{munilabeliana}(\ref{filiform})\\
\hline
\vbox to .6cm{}
$L_{6,19}(\epsilon)$ &  4 & $\left[
\begin{smallmatrix}
0 & -\frac{1}{\sqrt{-\epsilon}}x_2 + x_3 &  \frac{1}{\sqrt{-\epsilon}}x_4 - x_5  &  -\frac{2}{\sqrt{-\epsilon}}z_1 \\
0 &  0  & x_1 & x_4 + \sqrt{-\epsilon}x_5  \\
0 &  0  &  0  & x_2 + \sqrt{-\epsilon}x_3 \\
0& 0 & 0 & 0
\end{smallmatrix}
\right]$ &\begin{tabular}{l}
Pr.\ref{LemaL19}; if there \\
exists $\alpha \in \k^*$\\
such that \\
$\epsilon + \alpha^2 = 0$\\
\end{tabular}\\
\hline
\end{tabular}
$$

\begin{table}[!h]
\begin{centering}
$$
\begin{tabular}{||c|c|c||c||}
\hline
$\pmb{\g}$ & $\pmb{\mu}$ &  {\bf faithful nilrepresent.} & {\bf ref.}\\
\hline
\hline
\vbox to .9cm{}
$L_{6,19}(\epsilon)$        & 5 & $\left[
\begin{smallmatrix}
0 & x_1 & x_4 & x_5 & z_1 \\
0 & 0 & x_2 & x_3 & 0 \\
0 & 0 & 0 & 0 & -x_2 \\
0 & 0 & 0 & 0 & -\epsilon x_3 \\
0 & 0 & 0 & 0 & 0
\end{smallmatrix}
\right]$ & \begin{tabular}{l}
Pr.\ref{LemaL19}; if\\
$\epsilon + \alpha^2 \neq 0$\\
for all $\alpha \in \k^*$\\
\end{tabular}\\
\hline
\vbox to .9cm{}
$L_{6,20}$ & 5 & $\left[
\begin{smallmatrix}
0 & x_1 & 0 & x_4 & z_1\\
0 & 0 & x_1 & x_2 & x_5 \\
0 & 0 & 0 & 0 & x_3 \\
0 & 0 & 0 & 0 & -x_2 \\
0 & 0 & 0 & 0 & 0
\end{smallmatrix}
\right]$ & Eq.(\ref{remark:L19})\\
\hline
\vbox to .9cm{}
\multirow{4}{*}{$L_{6,21}(\epsilon)$} & 5 & $\left[
\begin{smallmatrix}
0 & -x_1 + x_2 & (\epsilon + 1)x_3 & -x_4-\epsilon x_5 & 3z_1\\
0 & 0 & x_1-(\epsilon + 2)x_2  & x_3 & -2x_4 + \epsilon x_5 \\
0 & 0 & 0 & x_2 & -x_3 \\
0 & 0 & 0 & 0 & 2x_2 + x_1\\
0 & 0 & 0 & 0 & 0
\end{smallmatrix}
\right]$ &\begin{tabular}{l}
Eq.(\ref{remark:L19}); \\
if $\epsilon \neq 0$\\
\end{tabular}\\
\cline{2-4}
\vbox to .9cm{}
        & 5 & $\left[
\begin{smallmatrix}
0 & -x_1 + x_2 + x_3& -x_3 - 2 x_4 -x_5 & x_4- x_5 & 3z_1\\
0 & 0 & x_2 & -x_3 & x_4 - x_5 \\
0 & 0 & 0 &x_1 + x_2 & -x_3 \\
0 & 0 & 0 & 0 &  x_1\\
0 & 0 & 0 & 0 & 0
\end{smallmatrix}
\right]$ & \begin{tabular}{l}
Eq.(\ref{remark:L19}); \\
if $\epsilon= 0$.\\
\end{tabular}\\
\hline
\vbox to .9cm{}
$L_{6,22}(\epsilon)$ & 5 & $\left[
\begin{smallmatrix}
0 & x_1 & x_4 & z_1 & z_2\\
0 & 0   &  0  & x_2 & x_3 \\
0 & 0   & 0   & x_3 & \epsilon x_2 \\
0 & 0 & 0  & 0 & 0  \\
0 & 0 & 0 & 0 & 0
\end{smallmatrix}
\right]$ & Eq.(\ref{remark:L19})\\
\hline
\vbox to .9cm{}
$L_{6,23}$ & 5 & $\left[
\begin{smallmatrix}
0 & x_1 & -x_4 & z_2 & z_1\\
0 & 0   & x_1  & 0   & x_3 \\
0 & 0   & 0    & x_1 & x_2 \\
0 & 0   & 0    & 0   & 0 \\
0 & 0   & 0    & 0   & 0
\end{smallmatrix}
\right]$ & Eq.(\ref{remark:L19})\\
\hline
\vbox to .9cm{}
\multirow{4}{*}{$L_{6,24}(\epsilon)$} & 5 & Remark \ref{faithL624} &\begin{tabular}{l}
Cor.\ref{cro:l624}; if there \\
exists $a \in \k$\\
such that \\
$a^2 - \epsilon = 0$\\
\end{tabular}\\
\cline{2-4}
\vbox to 1cm{}
        & 6 & $\left[
\begin{smallmatrix}
0 & x_2& x_1 & x_3& -2 z_1& -z_2\\
0 & 0  & 0   & 0  & -2x_4   & 0 \\
0 & 0  & 0   &x_2 & -x_3  & -\epsilon x_4 \\
0 & 0  & 0   & 0  &  x_1  & x_2\\
0 & 0  & 0   & 0  & 0     & 0 \\
0 & 0  & 0   & 0  & 0     & 0 
\end{smallmatrix}
\right]$ & \begin{tabular}{l}
Cor.\ref{cro:l624}; if\\
$a^2 -\epsilon \neq 0$\\
for all $a \in \k$\\
\end{tabular}\\
\hline
\vbox to .9cm{}
$L_{6,25}$ & 5 & $\left[
\begin{smallmatrix}
0 & x_1 & x_3  & 2 z_1  & z_2\\
0 & 0   & x_2  & x_3    & x_4 \\
0 & 0   & 0    & -x_1 & 0 \\
0 & 0 & 0  & 0 & 0  \\
0 & 0 & 0 & 0 & 0
\end{smallmatrix}
\right]$ & Eq.(\ref{remark:L19})\\
\hline
\vbox to .9cm{}
$L_{6,26}$ & 5 & $\left[
\begin{smallmatrix}
0 & 0 & x_1 & z_1& z_2\\
0 & 0 & x_2 & 0  & z_3 \\
0 & 0 & 0   & x_2 & x_3 \\
0 & 0 & 0  & 0 & 0  \\
0 & 0 & 0 & 0 & 0
\end{smallmatrix}
\right]$ & Eq.(\ref{remark:L19})\\
\hline
\end{tabular}
$$
\caption{Lie algebras such that $\mu= \mu_{nil}$}
\label{Tabla:8}
\end{centering}
\end{table}

\begin{remark}\label{faithL624}
For space reason, we give a minimal faithful representation of the Lie algebra
$L_{6,24}(\epsilon)$ (if there exists $a \in \k$ such that $a^2 - \epsilon = 0$) in this remark:
$$\left[
\begin{smallmatrix}
0 & \sqrt{\epsilon}x_1 + x_2 & (-\epsilon + \sqrt{\epsilon})x_4 + (3\epsilon -1)x_3 & (3\sqrt{\epsilon} - 1)z_1 + (3 -\sqrt{\epsilon})z_2  & (-\sqrt{\epsilon}+ 1)z_1 + (-\sqrt{\epsilon}+ 1)z_2\\
0 & 0 & x_1 + x_2  & (-\epsilon + 4\sqrt{\epsilon} - 1)x_4 + 2x_3 & (-\epsilon + 1)x_4 \\
0 & 0 & 0 & -x_1 + x_2 & x_1 + x_2 \\
0 & 0 & 0 & 0 & 0\\
0 & 0 & 0 & 0 & 0
\end{smallmatrix}
\right].$$
\end{remark}


\subsection{$\mu \text{ and } \mu_{nil}$ of $L_{6,19}(\epsilon)$ for $\epsilon \in \k$}\label{sec:L619}


A basis of the Lie algebra
$L_{6,19}(\epsilon)$ is
$B=\{X_1, X_2, X_3, X_4, X_5, Z_1\}$ such that the brackets are
\begin{equation}\label{eq:6}
[X_1, X_2] = X_4,\; [X_1, X_3] = X_5,\;  [X_2, X_4] = Z_1,\; [X_3, X_5] = \epsilon Z_1,
\end{equation}
(see \S \ref{sec:dimen6}). The Lie algebra
$L_{6,19}(\epsilon) \text{ and } L_{6,19}(\delta)$ are isomorphic if and only if there exist
$\alpha \in \k^*$ such that
\begin{equation}\label{eq:delta}
\epsilon = \alpha^2 \delta.
\end{equation}

\begin{proposition}\label{LemaL19}
Let
$\k$ be a field of characteristic
$\neq 2$ then
$$\mu_{nil}(L_{6,19}(\epsilon)) =
         \begin{cases}
         4, & \text{ if } \exists \alpha \in \k^* : \epsilon + \alpha^2 = 0; \\
         5, & \text{ if } \forall \alpha \in \k^* : \epsilon + \alpha^2 \neq 0;
         \end{cases}\text{ }$$
and
$\mu(L_{6,19}(\epsilon))= \mu_{nil}(L_{6,19}(\epsilon))$ for all
$\epsilon \in \k$.
\end{proposition}

First, we prove some useful lemmas.

We define the linear mapping
$\pi_1: L_{6,19}(\epsilon) \rightarrow \mathfrak{gl}(5)$ given by
$$\pi_1\left(\sum_{i=1}^5 x_iX_i + z_1Z_1\right) = \left[
\begin{smallmatrix}
0 & x_1 & x_4 & x_5 & z_1 \\
0 & 0 & x_2 & x_3 & 0 \\
0 & 0 & 0 & 0 & -x_2 \\
0 & 0 & 0 & 0 & -\epsilon x_3 \\
0 & 0 & 0 & 0 & 0
\end{smallmatrix}
\right].$$

\begin{lemma}\label{L6191}
Let
$\epsilon \in \k$ then
$(\pi_1, \k^5)$ is a faithful nilrepresentation of
$L_{6,19}(\epsilon)$.
\end{lemma}

\begin{proof}
We first to prove that
$\pi_1$ is a representation of
$L_{6,19}(\epsilon)$. Let
$B= \{X_1, X_2, X_3, X_4, X_5, Z_1\}$ be a basis such that verified the Equation (\ref{eq:6}). It follows that
$X = \sum_{i=1}^5 x_i X_i + z_1 Z_1, \;Y = \sum_{i=1}^5 y_i X_i + \widetilde{z}_1 Z_1 \in L_{6,19}(\epsilon)$ for all
$X, Y \in L_{6,19}(\epsilon)$.
By Equation (\ref{eq:6}), we have
\begin{equation}\label{eq:7}
[X, Y] = (x_1 y_2 - y_1 x_2) X_4 + (x_1 y_3 - y_3 x_1) X_5 + (x_2 y_4 - y_2 x_4) Z_1 + \epsilon (x_3 y_5 - y_5 x_3) Z_1.
\end{equation}
By the other hand,
$$\pi(X)= \left[
         \begin{smallmatrix}
         0 & x_1 & x_4 & x_5 & z_1 \\
         0 & 0 & x_2 & x_3 & 0 \\
         0 & 0 & 0 & 0 & -x_2 \\
         0 & 0 & 0 & 0 & -\epsilon x_3 \\
         0 & 0 & 0 & 0 & 0
\end{smallmatrix}
         \right] \text{ and }
\pi(Y)= \left[
         \begin{smallmatrix}
         0 & y_1 & y_4 & y_5 & \widetilde{z}_1 \\
         0 & 0 & y_2 & y_3 & 0 \\
         0 & 0 & 0 & 0 & -y_2 \\
         0 & 0 & 0 & 0 & -\epsilon y_3 \\
         0 & 0 & 0 & 0 & 0
\end{smallmatrix}
         \right].
$$
By straightforward calculation we have
$[\pi(X), \pi(Y)]$ coincides with the Equation (\ref{eq:7}). Hence
$(\pi_1, \k^5)$ is a representation of the Lie algebra
$L_{6,19}(\epsilon)$. It is clear that
$\pi_1$ is injective and it complete the proof.
\end{proof}

If
$\epsilon = -1$, we define the linear mapping
$\pi_2: L_{6,19}(-1) \rightarrow \mathfrak{gl}(5)$ given by
$$\pi_2\left(\sum_{i=1}^5 x_iX_i + z_1Z_1\right) =
\left[
\begin{smallmatrix}
0 & x_2 + x_3 & -x_4 - x_5 & 2 z_1 \\
0 &      0    &   x_1      & x_4 - x_5 \\
0 & 0 & 0 & x_2 - x_3 \\
0 & 0 & 0 & 0 \\
\end{smallmatrix}
\right]
$$

\begin{lemma}\label{L6192}
Let
$\k$ be a field of characteristic
$\neq 2$ then $(\pi_2, \k^4)$ is a faithful representation of
$L_{6,19}(-1)$.
\end{lemma}

\begin{proof}
The proof is similar to the proof of Lemma \ref{L6191}.
\end{proof}

\begin{proof}$[\text{Proof of Proposition \ref{LemaL19}.}]$
By Lemma \ref{L6191}, we have
$
\mu_{nil}(L_{6,19}(\epsilon)) \leq 5,
$
for every
$\epsilon \in \k$.

But
$\mu_{nil}(L_{6,19}(-1)) \leq 4$, which follows from Lemma \ref{L6192} and  by Equation (\ref{eq:cotainf}), we obtain
$$\mu_{nil}(L_{6,19}(-1)) = 4.$$
Hence, by Engel's Theorem
$L_{6,19}(-1) \cong \mathfrak{n}_4(\k)$ and from Equation (\ref{eq:delta}) we conclude that
$L_{6,19}(\epsilon) \cong \mathfrak{n}_4(\k)$ if and only if there exist
$\alpha \in \k^*$ such that
$\epsilon + \alpha^2 = 0$.
Therefore, if
$\epsilon + \alpha^2 \neq 0$ for all
$\alpha \in \k^*$, we get
$5 \leq \mu_{nil}(L_{6,19}(\epsilon))$.

Since
$\mathfrak{z}(L_{6,19}(\epsilon)) \subseteq [L_{6,19}(\epsilon), L_{6,19}(\epsilon)]$ for every
$\epsilon \in \k$ and by Corollary \ref{coro:mumunil} we have
$\mu(L_{6,19}(\epsilon)) = \mu_{nil}L_{6,19}(\epsilon)$. This complete the proof.
\end{proof}

From Equation (\ref{eq:cotainf}), we have
$\mu_{nil}(L_{6,j}) \geq 4$. Suppose
$\mu_{nil}(L_{6,j})= 4$, it is known that
$L_{6,j} \cong \mathfrak{n}_4(\k)$, by Engel's Theorem. Now we combine this and Proposition \ref{LemaL19} to obtain
$L_{6,j} \cong L_{6,19}(\epsilon)$ for some
$\epsilon  \in \k$. Then
$j= 19$, which follows from \cite{Gr}. Thus
\begin{equation}\label{remark:L19}
5 \leq \mu_{nil}(L_{6,j}).
\end{equation}
for all
$j \neq 19$.


\subsection{The lower bound of $\mu_{nil}$ for $L_{6,9} \text{ and } L_{6,24}(\epsilon)$ for all $\epsilon \in \k$}\label{sec:L6924}


The main object of this subsection is to proof the following theorems.

\begin{theorem}\label{L69}
Let
$n \in \mathbb{N}$ and let
$\g$ be a Lie subalgebra of
$\mathfrak{n}_n(\k)$ isomorphic to
$L_{6,9}$. Then
$n \geq 6$.
\end{theorem}

\begin{theorem}\label{L624}
Let
$\epsilon \in \k, n \in \mathbb{N}$ and let
$\g$ be a Lie subalgebra of
$\mathfrak{n}_n(\k)$ isomorphic to
$L_{6,24}(\epsilon)$. Then
\begin{enumerate}[(1)]
\item $n = 5$ if  there exist
      $a \in \k$ such that
      $a^2 - \epsilon = 0$ and ;
\item $n \geq 6$ in otherwise.
\end{enumerate}
\end{theorem}

By Theorems \ref{L69} and \ref{L624}, we easily obtain the following Corollaries, respectively.

\begin{corollary}\label{coro:L69}
$\mu_{nil}(L_{6,9})\geq 6$
\end{corollary}

\begin{corollary}\label{cro:l624}
Let
$\epsilon \in \k$ then
$$\mu_{nil}(L_{6,24}(\epsilon))\geq
\begin{cases}
5, \text{ } \exists a \in \k : a^2 - \epsilon = 0;\\
6, \text{ }  \text{ in otherwise }.
\end{cases}
$$
\end{corollary}

Let
$B = \{X_1, X_2, X_3, X_4, Z_1, Z_2\}$ be a basis of
$L_{6,24}(\epsilon)$ such that the only non-zero brackets are
$$
[X_1, X_2] = X_3, [X_1, X_3] = Z_1, [X_2, X_3] = Z_2, [X_1, X_4] = \epsilon Z_2, [X_2, X_4] = Z_1.
$$
(see \S \ref{sec:dimen6}).

\begin{remark}\label{remark:L59}
The Lie algebra
$\h$ generated by the set
$\{X_1, X_2, X_3, Z_1, Z_2\}$ verified
\begin{equation}\label{eq:8}
[X_1, X_2] = X_3, [X_1, X_3] = Z_1, [X_2, X_3] = Z_2, [X_i, Z_j] = 0
\end{equation}
for
$i=1,2,3; j=1,2$ is a Lie algebra isomorphic to
$L_{5,9}$.
\end{remark}

In order to prove Theorems \ref{L69} and \ref{L624} we need the following results, respectively.

\begin{proposition}\label{lema:3}
Let
$\{X_1, X_2, X_3, Z_1, Z_2\} \subseteq \mathfrak{n}_5(\k)$ be a set linearly independent such that verified Equation
(\ref{eq:8}). Let
$X_4 \in \mathfrak{n}_5(\k)$ such that
$[X_i, X_4]= 0$ for
$i=1, 2, 3$ then
$X_4 \in \k\{X_1, X_2, X_3, Z_1, Z_2\}$.
\end{proposition}

\begin{proposition}\label{lema:4}
Let
$\epsilon \in \k$ and let
$\{X_1, X_2, X_3, Z_1, Z_2\} \subseteq \mathfrak{n}_5(\k)$ be a set linearly independent such that verified Equation
(\ref{eq:8}). Let
$X_4 \in \mathfrak{n}_5(\k)$ such that
\begin{enumerate}[(1)]
\item $\{X_1, X_2, X_3, X_4, Z_1, Z_2\}$ is a linearly independent and
\item $[X_1, X_4] = \epsilon Z_2 \;\; [X_2, X_4] = Z_1 \;\; [X_3, X_4] = [Z_j, X_4] =0$ for
      $j=1, 2$.
\end{enumerate}
Then there exists
$a \in \k$ such that
$a^2 - \epsilon = 0$.
\end{proposition}

\begin{proof}$[\text{Proof of Theorem \ref{L69}.}]$
Since
$\g$ is a Lie algebra isomorphic to
$L_{6,9}$, we have
\begin{enumerate}[(a)]
\item $n \geq 5$, by Equation (\ref{remark:L19}) and
\item there exists
$\{X_1, X_2, X_3, X_4, Z_1, Z_2\}$ a basis of
$\g$ that verified equation (\ref{eq:8}) and
$[X_i, X_4]= [Z_j, X_4]= 0$ for
$i= 1, 2,3; j= 1, 2$.
\end{enumerate}
By Proposition \ref{lema:3}, we obtain
$n \geq 6$.
\end{proof}

\begin{proof}$[\text{Proof of Theorem \ref{L624}.}]$
Since
$\g$ is a Lie algebra isomorphic to
$L_{6,24}(\epsilon)$, we have
\begin{enumerate}[(a)]
\item $n \geq 5$, by equation \ref{remark:L19} and
\item there exists
      $\{X_1, X_2, X_3, X_4, Z_1, Z_2\}$ a basis of
      $\g$ that verified equation (\ref{eq:8}) and
      $[X_1, X_4] = \epsilon Z_2 \;\; [X_2, X_4] = Z_1 \;\; [X_4, Z_j]= 0$ for
      $j=1, 2$.
\end{enumerate}
If
$n = 5$, by Proposition \ref{lema:4}, there exist
$a \in \k$ such that
$a^2 - \epsilon = 0$. It complete the proof.
\end{proof}

The remaining of this section is devoted to prove the Propositions \ref{lema:3} and \ref{lema:4}.

Since Lie algebra
$L_{6,9} = L_{5,9} \oplus L_{1,1}$ (see \S \ref{sec:dimen6}), we get
$L_{5,9}$ is a Lie subalgebra of
$L_{6,9}$ of codimension
$1$, as in the case
$L_{6,24}(\epsilon)$ for every
$\epsilon \in \k$. Then, if
$(\pi, V)$ is a faithful representation of
$L_{6,9}$, or
$L_{6,24}(\epsilon)$, we have
$(\pi\mid_{L_{5,9}}, V)$ is a faithful representation of
$L_{5,9}$. Hence, in order to prove Propositions \ref{lema:3} and \ref{lema:4}, we need
the following results that gives some precise information about the structure of a
faithful nilrepresentation
$(\pi, \k^5)$ of
$L_{5,9}, L_{6,9} \text{ and } L_{6,24}(\epsilon)$.

\begin{lemma}\label{lemabase}
Let
$V$ be a vector space over
$\k$ of dimension 3 and let
$W$ be a subspace of
$V$ of dimension
$2$. If
$B=\{v_1, v_2, v_3\}$ is a basis of
$V$ then
$W$ has a basis of any of the following ways:
\begin{enumerate}[a)]
\item $\{v_1 + a v_3, v_2 + b v_3\}$ with
      $a, b \in \k$;
\item $\{v_1 + c v_2, v_3\}$ with
      $c \in \k$;
\item $\{v_2, v_3\}$.
\end{enumerate}
\end{lemma}

\begin{proof}
The proof is an straightforward calculation of the linear algebra.
\end{proof}

Let
$\{ E_{ij}\}$ be the canonical basis of
$\mathfrak{gl}(\k^n)$.

\begin{lemma}\label{lema:1}
Let
$\h$ be a Lie subalgebra of
$\mathfrak{n}_5(\k)$ isomorphic to
$L_{5,9}$. Then the center
$\mathfrak{z}(\h)$ is of any of the following ways
\begin{enumerate}[(1)]
\item $\mathfrak{z}(\h)= \operatorname{span}_{\k} \{E_{14} + cE_{25}, E_{15}\}$ with
      $c \in \k$;
\item $\mathfrak{z}(\h)= \operatorname{span}_{\k} \{E_{25}, E_{15}\}$.
\end{enumerate}
\end{lemma}

\begin{proof}
Since
$\h$ is a Lie subalgebra of
$\mathfrak{n}_5(\k)$ isomorphic to
$L_{5,9}$, we have
$\mathfrak{z}(\h) = [\h, [\h, \h]]$ and
$\dim \mathfrak{z}(\h) = 2$. Hence
$\mathfrak{z}(\h) \subseteq \operatorname{span}_{\k} \{E_{14}, E_{25}, E_{15}\}$. By Lemma \ref{lemabase} a basis of
$\mathfrak{z}(\h)$ is of any of the following ways:
\begin{enumerate}[-]
\item $\{E_{14} + a E_{15}, E_{25} + b E_{15}\}$ with
      $a, b \in \k$;
\item $\{E_{14} + c E_{25}, E_{15}\}$ with
      $c \in \k$;
\item $\{E_{25}, E_{15}\}$.
\end{enumerate}
Suppose that the center is
$\mathfrak{z}(\h) = \operatorname{span}_{\k} \{E_{14} + a E_{15}, E_{25} + b E_{15}\}$. It follows that
$$\h \subseteq \left\{ \left[
      \begin{smallmatrix}
      0 & 0 & x_{13} & x_{14} & x_{15} \\
      0 & 0 & x_{23} & x_{24} & x_{25} \\
      0 & 0 & 0 & x_{34} & x_{35} \\
      0 & 0 & 0 & 0 & 0\\
      0 & 0 & 0 & 0 & 0
      \end{smallmatrix}
      \right]: x_{ij} \in \k \right\}.$$
Therefore
$\h$ is
$k$-step nilpotent Lie algebra with
$k \leq 2$, which contradicts that
$\h$ is isomorphic to
$L_{5,9}$. The proof is now completed.
\end{proof}

\begin{lemma}\label{lema:2}
Let
$\h$ be a Lie subalgebra of
$\mathfrak{n}_5(\k)$ isomorphic to
$L_{5,9}$. Let
$X_1= [x_{ij}], X_2= [y_{ij}] \in \mathfrak{n}_5(\k)$
such that the set
$$\{X_1, X_2, X_3, Z_1, Z_2\} \subseteq \mathfrak{n}_5(\k)$$
is a basis of
$\h$ that verified Equation (\ref{eq:8}).
\begin{enumerate}[(1)]
\item If
      $\mathfrak{z}(\h) = \operatorname{span}_{\k} \{E_{14} + cE_{25}, E_{15}\}$ with
      $c \neq 0$ then
      $y_{34} \neq 0, x_{23}= \frac{y_{23}}{y_{34}} x_{34}$ and
      $y_{23}\left(x_{12}y_{34}-x_{34}y_{12}\right)\neq 0$;
\item if
      $\mathfrak{z}(\h) = \operatorname{span}_{\k} \{E_{14}, E_{15}\}$ then
      $\left(y_{{12}}x_{23} -x_{{12}}y_{{23}}\right) \left(x_{{34}}y_{{35}}-x_{{35}}y_{{34}} \right)\neq 0$;
\item if
      $\mathfrak{z}(\h) = \operatorname{span}_{\k} \{E_{25}, E_{15}\}$ then
      $\left(x_{23}y_{{12}} -x_{{12}}y_{{23}}\right)\left(x_{{34}}y_{{35}}-x_{{35}}y_{{34}}\right)\neq 0$.
\end{enumerate}
\end{lemma}

\begin{proof}
Since
$X_3= [X_1, X_2], Z_1= [X_1, X_3] \text{ and }  Z_2= [X_2, X_3]$, we get
$$
Z_1 = \left[
      \begin{smallmatrix}
      0 & 0 & 0 & a_1 & a_2 \\
      0 & 0 & 0 & 0 &  a_3 \\
      0 & 0 & 0 & 0 & 0 \\
      0 & 0 & 0 & 0 & 0 \\
      0 & 0 & 0 & 0 & 0
      \end{smallmatrix}
      \right] \text{ and }
Z_2 = \left[
      \begin{smallmatrix}
      0 & 0 & 0 & b_1 & b_2 \\
      0 & 0 & 0 & 0 &  b_3 \\
      0 & 0 & 0 & 0 & 0 \\
      0 & 0 & 0 & 0 & 0 \\
      0 & 0 & 0 & 0 & 0
      \end{smallmatrix}
      \right]
$$
with
$$
\begin{cases}
a_1= x_{{12}} \left( x_{{23}}y_{{34}}-y_{{23}}x_{{34}} \right) -
       \left( x_{{12}}y_{{23}}-y_{{12}}x_{{23}}\right) x_{{34}},\\
a_2 = x_{{12}} \left( x_{{23}}y_{{35}}+x_{{24}}y_{{45}}-y_{{23}}x_{{35}}-y_{{24}}x_{{45}} \right) +
             x_{{13}} \left( x_{{34}}y_{{45}}-y_{{34}}x_{{45}} \right) +\\
            \;\;\;\;\;\;\;\;- x_{{35}}\left( x_{{12}}y_{{23}}-y_{{12}}x_{{23}} \right) - x_{{45}} \left( x_{{12}}y_{{24}}+x_{{13}}y_{{34}}-y_{{12}}x_{{24}}-y_{{13}}x_{{34}} \right),\\
a_3= x_{{23}} \left( x_{{34}}y_{{45}}-y_{{34}}x_{{45}} \right) -  x_{{45}}\left( x_{{23}}y_{{34}}-y_{{23}}x_{{34}} \right);\\
b_1= y_{{12}} \left( x_{{23}}y_{{34}}-y_{{23}}x_{{34}} \right) -
        y_{{34}}\left( x_{{12}}y_{{23}}-y_{{12}}x_{{23}} \right),\\
b_2 = y_{{12}} \left( x_{{23}}y_{{35}}+x_{{24}}y_{{45}}-y_{{23}}x_{{35}}-y_{{24}}x_{{45}} \right) +
      y_{{13}} \left( x_{{34}}y_{{45}}-y_{{34}}x_{{45}} \right)+\\
\;\;\;\;\;\;\;\;- y_{{35}}\left( x_{{12}}y_{{23}}-y_{{12}}x_{{23}} \right)- y_{{45}}\left( x_{{12}}y_{{24}}+x_{{13}}y_{{34}}-y_{{12}}x_{{24}}-y_{{13}}x_{{34}} \right),\\
b_3= y_{{23}} \left( x_{{34}}y_{{45}}-y_{{34}}x_{{45}} \right) - y_{{45}} \left( x_{{23}}y_{{34}}-y_{{23}}x_{{34}} \right).
\end{cases}
$$

\noindent Case $(1)$. Let
$c \neq 0$, if
$\mathfrak{z}(\h) = \operatorname{span}_{\k}\{E_{14} + cE_{25}, E_{15}\}$ we obtain
$$\h \subseteq \left\{ \left[
      \begin{smallmatrix}
      0 & m_{12} & m_{13} & m_{14} & m_{15} \\
      0 & 0 & m_{23} & m_{24} & m_{25} \\
      0 & 0 & 0 & m_{34} & m_{35} \\
      0 & 0 & 0 & 0 & c m_{12} \\
      0 & 0 & 0 & 0 & 0
      \end{smallmatrix}
      \right]: m_{ij} \in \k\right\}.$$ Therefore
\begin{enumerate}[(a)]
\item $x_{45}= c x_{12}, y_{45}= c y_{12}$ and
\item $Z_1, Z_2 \in \operatorname{span}_{\k}\{E_{14} + cE_{25}, E_{15}\}$.
\end{enumerate}
By
$(a)$, we get
$$
\begin{cases}
a_1 = x_{{12}} \left(x_{{23}}y_{{34}}-y_{{23}}x_{{34}} \right) -  x_{{34}} \left( x_{{12}}y_{{23}}-y_{{12}}x_{{23}} \right),\\
a_3 = c\left(x_{{23}} \left(x_{{34}}y_{{12}}-y_{{34}}x_{{12}} \right) -  x_{{12}}\left( x_{{23}}y_{{34}}-y_{{23}}x_{{34}} \right)\right),\\
b_1 = y_{{12}} \left(x_{{23}}y_{{34}}-y_{{23}}x_{{34}} \right) -  y_{{34}} \left( x_{{12}}y_{{23}}-y_{{12}}x_{{23}} \right),\\
b_3 = c\left(y_{{23}} \left(x_{{34}}y_{{12}}-y_{{34}}x_{{12}} \right) -  y_{{12}} \left( x_{{23}}y_{{34}}-y_{{23}}x_{{34}} \right)\right).
\end{cases}
$$
Since
$c \neq 0$ and by
$(b)$, we have
$a_3= c a_1; b_3= c a_1$, it follows that
$$
\begin{cases}
 x_{{12}} \left( x_{{23}}y_{{34}} - y_{23}x_{34} \right) &= 0, \\
 y_{{12}} \left(x_{{23}}y_{{34}} - y_{{23}}x_{{34}} \right)&=0.
\end{cases}
$$
By straightforward calculation we have
the only solution that makes
$\left\{Z_1, Z_2\right\}$ is linearly independent is
$y_{34} \neq 0 \text{ and } x_{23} = \frac{y_{23}}{y_{34}}x_{34}$ and since
$$Z_1 = \left[
      \begin{smallmatrix}
      0 & 0 & 0 & a_1 & a_2 \\
      0 & 0 & 0 & 0 &  c a_1 \\
      0 & 0 & 0 & 0 & 0 \\
      0 & 0 & 0 & 0 & 0 \\
      0 & 0 & 0 & 0 & 0
      \end{smallmatrix}
      \right] \text{ and }
Z_2 = \left[
      \begin{smallmatrix}
      0 & 0 & 0 & b_1 & b_2 \\
      0 & 0 & 0 & 0 &  c b_1 \\
      0 & 0 & 0 & 0 & 0 \\
      0 & 0 & 0 & 0 & 0 \\
      0 & 0 & 0 & 0 & 0
      \end{smallmatrix}
      \right],
$$
we have
$\det \left[
     \begin{smallmatrix}
     a_1 & a_2 \\
     b_1 & b_2
     \end{smallmatrix}
     \right] \neq 0$ with
$$
\begin{cases}
      a_2 = x_{{12}} \left( x_{{23}}y_{{35}}-y_{{23}}x_{{35}}+ c x_{{24}}y_{{12}}- c y_{{24}}x_{{12}} \right)+
      c x_{{13}} \left( x_{{34}}y_{{12}}-y_{{34}}x_{{12}} \right) +\\
      \;\;\;\;\;\;\;\;-  x_{{35}}\left( x_{{12}}y_{{23}}-y_{{12}}x_{{23}} \right)- cx_{{12}}\left( x_{{12}}y_{{24}}+x_{{13}}y_{{34}}-y_{{12}}x_{{24}}-y_{{13}}x_{{34}} \right),\\
b_2 = y_{{12}} \left( x_{{23}}y_{{35}}-y_{{23}}x_{{35}}+ c x_{{24}}y_{{12}}- cy_{{24}}x_{{12}} \right) +
      c y_{{13}} \left( x_{{34}}y_{{12}}-y_{{34}}x_{{12}} \right) + \\
      \;\;\;\;\;\;\;\;- y_{{35}} \left( x_{{12}}y_{{23}}-y_{{12}}x_{{23}} \right)- c y_{{12}} \left(x_{{12}}y_{{24}}+x_{{13}}y_{{34}}-y_{{12}}x_{{24}}-y_{{13}}x_{{34}} \right).
\end{cases}
$$
Therefore
$$
\det \left[
     \begin{smallmatrix}
     a_1 & a_2 \\
     b_1 & b_2
     \end{smallmatrix}
     \right]= {\frac {-2 y_{{23}} \left(y_{{34}}x_{{12}}- x_{{34}}y_{{12}}\right)^{2} a}{{y_{{34}}}^{2}}}\neq 0
$$
with
$a=c y_{{34}} \left(y_{{24}}x_{{12}} + x_{{13}}{y_{{34}}}
- x_{{34}}y_{{13}}-x_{{24}}y_{{12}} \right) + y_{{23}}\left(x_{{35}}y_{{34}}- x_{{34}}y_{{35}}\right)$.

\noindent Case $(2)$. If
$c = 0$ and
$\mathfrak{z}(\h) = \operatorname{span}_{\k} \{E_{14}, E_{15}\}$, we get
$$\h \subseteq \left\{ \left[
      \begin{smallmatrix}
      0 & m_{12} & m_{13} & m_{14} & m_{15} \\
      0 & 0 & m_{23} & m_{24} & m_{25} \\
      0 & 0 & 0 & m_{34} & m_{35} \\
      0 & 0 & 0 & 0 & 0 \\
      0 & 0 & 0 & 0 & 0
      \end{smallmatrix}
      \right]: m_{ij} \in \k\right\}.$$ Hence
$x_{45}= y_{45}= 0$, thus
$a_3= b_3= 0$. Since the set
$\left\{Z_1, Z_2\right\}$ is linearly independent, we obtain
$\det \left[
     \begin{smallmatrix}
     a_1 & a_2 \\
     b_1 & b_2
     \end{smallmatrix}
     \right] \neq 0$ with
$$
\begin{cases}
a_2 = x_{{12}} \left(x_{{23}}y_{{35}}-y_{{23}}x_{{35}}\right)- x_{{35}} \left( x_{{12}}y_{{23}}-y_{{12}}x_{{23}} \right),\\
b_2 = y_{{12}} \left(x_{{23}}y_{{35}}-y_{{23}}x_{{35}}\right) - y_{{35}} \left( x_{{12}}y_{{23}}-y_{{12}}x_{{23}} \right).
\end{cases}
$$
It follows that
$
\det \left[
     \begin{smallmatrix}
     a_1 & a_2 \\
     b_1 & b_2
     \end{smallmatrix}
     \right] = 2\left(y_{{12}}x_{23} -x_{{12}}y_{{23}}\right) ^{2} \left( x_{{34}}y_{
{35}}-x_{{35}}y_{{34}} \right) \neq 0.
$

\noindent Case $(3)$. Analysis similar to that in the proof of Case $(2)$ shows that
$$\left(y_{{12}}x_{23} -x_{{12}}y_{{23}}\right) ^{2} \left( x_{{34}}y_{{35}}-x_{{35}}y_{{34}} \right)\neq 0.$$
The proof is now completed.
\end{proof}

\begin{lemma}\label{lema:5}
Let
$\{X_1, X_2, X_3, Z_1, Z_2\} \subseteq \mathfrak{n}_5(\k)$ be a set linearly independent that verified Equation (\ref{eq:8}). Let
$X_4= [a_{ij}] \in \mathfrak{n}_5(\k)$ be non-zero matrix such that
$[X_3, X_4] = 0$.
\begin{enumerate}[(1)]
\item If
      $Z_1, Z_2, [X_1, X_4], [X_2, X_4] \in \operatorname{span}_{\k} \{E_{14} + cE_{25}, E_{15}\}$ and
      $a_{45}= c a_{12}$ then
      $a_{12}= a_{23}= a_{34}= 0$ and
      $a_{35}= -c a_{13}\frac{y_{34}}{y_{23}}$;
\item if
      $Z_1, Z_2, [X_1, X_4], [X_2, X_4] \in \operatorname{span}_{\k} \{E_{25}, E_{15}\}$ and
      $a_{12}= 0$ then
      $a_{23}= a_{34}= a_{45}= 0$.
\end{enumerate}
\end{lemma}

\begin{proof}
Let
$\h$ be a Lie algebra generated by
$\{X_1, X_2, X_3, Z_1, Z_2\}$. By Remark \ref{remark:L59}, we have
$\h$ is a Lie subalgebra of
$\mathfrak{n}_5(\k)$ isomorphic to
$L_{5,9}$ and the center is
$\mathfrak{z}(\h) = \operatorname{span}_{\k} \{Z_1, Z_2\}$.

\noindent Case $(1)$. If
$Z_1, Z_2 \in \k \{E_{14} + cE_{25}, E_{15}\}$, we get
$$\h \subseteq \left\{ \left[
      \begin{smallmatrix}
      0 & m_{12} & m_{13} & m_{14} & m_{15} \\
      0 & 0 & m_{23} & m_{24} & m_{25} \\
      0 & 0 & 0 & m_{34} & m_{35} \\
      0 & 0 & 0 & 0 & c m_{12} \\
      0 & 0 & 0 & 0 & 0
      \end{smallmatrix}
      \right]: m_{ij} \in \k\right\}$$
for every
$c \in \k$.
We first assume that
$c \neq 0$. By Lemma \ref{lema:2} in the case (1), we get
$$
X_1= \left[
      \begin{smallmatrix}
      0 & x_{12} & x_{13} & x_{14} & x_{15} \\
      0 & 0 & \frac{y_{23}}{y_{34}}x_{34} & x_{24} & x_{25} \\
      0 & 0 & 0 & x_{34} & x_{35} \\
      0 & 0 & 0 & 0 & cx_{12} \\
      0 & 0 & 0 & 0 & 0
      \end{smallmatrix}
      \right],
X_2= \left[
      \begin{smallmatrix}
      0 & y_{12} & y_{13} & y_{14} & y_{15} \\
      0 & 0 & y_{23} & y_{24} & y_{25} \\
      0 & 0 & 0 & y_{34} & y_{35} \\
      0 & 0 & 0 & 0 &  cy_{12} \\
      0 & 0 & 0 & 0 & 0
      \end{smallmatrix}
      \right]$$ and
\begin{equation}\label{eq:9}
y_{23}\left(x_{12}y_{34}-x_{34}y_{12}\right)\neq 0.
\end{equation}
Since
$0= [X_3, X_4]= [[X_1, X_2], X_4]$, we obtain
$$\begin{cases}
(a) & a_{{34}}\frac{y_{{23}}}{y_{34}}(y_{{34}}x_{{12}} - x_{{34}}y_{{12}})= 0,\\
(b) & a_{{23}}c \left(y_{{34}}x_{{12}} - x_{{34}}y_{{12}} \right)= 0,\\
(c) & a_{{35}}\frac{y_{{23}}}{y_{34}}(y_{{34}}x_{{12}} - x_{{34}}y_{{12}}) + a_{{13}} c\left(y_{{34}}x_{{12}}- x_{{34}}y_{{12}}\right) - a_{12} m=0
      \text{ with }\\
    &m = 2c(x_{12}y_{24} - y_{12}x_{24}) + c(x_{13}y_{34}- y_{13}x_{34}) + \frac{y_{23}}{y_{34}}(x_{35}y_{34} - x_{34}y_{35}).
\end{cases}
$$
Since
$[X_2, X_4] \in \operatorname{span}_{\k} \{E_{14} + cE_{25}, E_{15}\}$, we have
\begin{enumerate}[(d)]
\item $y_{{12}}a_{{23}}-a_{{12}}y_{{23}} =0$.
\end{enumerate}
By $(a), (b), (d)$ and Equation (\ref{eq:9}), we get
$a_{34}= a_{23}= a_{12}= 0$. Finally, by $(c)$ and Equation (\ref{eq:9}), it follows that
$a_{{35}}\frac{y_{{23}}}{y_{34}} + c a_{{13}}= 0$.

Now assume
$c= 0$. By Lemma \ref{lema:2} in the case $(2)$, we obtain
\begin{equation}\label{eq:2}
\left(y_{{12}}x_{23} -x_{{12}}y_{{23}}\right) ^{2} \left( x_{{34}}y_{{35}}-x_{{35}}y_{{34}} \right) \neq 0.
\end{equation}
Since
$[X_1, X_4], [X_2, X_4] \in \operatorname{span}_{\k}  \{E_{14}, E_{15}\}$, we have the following equation
$$\begin{cases}
     \setlength{\unitlength}{13.10pt}
      \begin{picture}(20,3.5)
      \put(0.1,3){$x_{{12}}a_{{23}}-a_{{12}}x_{{23}}= 0$,}          \put(10,3){$y_{{12}}a_{{23}}-a_{{12}}y_{{23}}=0$,}
      \put(0.1,1.5){$x_{{23}}a_{{34}}-a_{{23}}x_{{34}}= 0$,}        \put(10,1.5){$y_{{23}}a_{{34}}-a_{{23}}y_{{34}}= 0$,}
      \put(0.1,0){$x_{{23}}a_{{35}}-a_{{23}}x_{{35}}= 0$,}          \put(10,0){$y_{{23}}a_{{35}}-a_{{23}}y_{{35}}= 0$.}
      \end{picture}
\end{cases}
$$
By Equation (\ref{eq:2}), we get
$a_{12}= a_{23}= a_{34}= a_{35}= 0$.

\noindent Case $(2)$. Analysis similar to that in the Case $(1)$ with
$c = 0$, shows that
$a_{13}= a_{23}= a_{34}= 0$ and the proof is complete.
\end{proof}

We are now in position to prove the Propositions \ref{lema:3} and \ref{lema:4}.

\begin{proof}
$[\text{Proof of Proposition \ref{lema:3}}.]$
Let
$\h$ be the Lie algebra generated by
$\{X_1, X_2, X_3, Z_1, Z_2\}.$ By Remark \ref{remark:L59}, we have
$\h$ is a Lie subalgebra of
$\mathfrak{n}_5(\k)$ isomorphic to
$L_{5,9}$ and by Lemma \ref{lema:1}, we have the center
$\mathfrak{z}(\h)$ is of any of the following ways
\begin{enumerate}[(1)]
\item $\mathfrak{z}(\h)= \operatorname{span}_{\k}  \{E_{14} + cE_{25}, E_{15}\}$ with
      $c \in \k$;
\item $\mathfrak{z}(\h)= \operatorname{span}_{\k}  \{E_{25}, E_{15}\}$.
\end{enumerate}
We first assume the case $(1)$. If
$c \neq 0$, by Lemma \ref{lema:2}(1) and Lemma \ref{lema:5} in the case $(1)$, we get
$$
X_1 = \left[
      \begin{smallmatrix}
      0 & x_{12} & x_{13} & x_{14} & x_{15} \\
      0 & 0 & \frac{y_{23}}{y_{34}}x_{34} & x_{24} & x_{25} \\
      0 & 0 & 0 & x_{34} & x_{35} \\
      0 & 0 & 0 & 0 & c x_{12} \\
      0 & 0 & 0 & 0 & 0
      \end{smallmatrix}
      \right],
X_2 = \left[
      \begin{smallmatrix}
      0 & y_{12} & y_{13} & y_{14} & y_{15} \\
      0 & 0 & y_{23} & y_{24} & y_{25} \\
      0 & 0 & 0 & y_{34} & y_{35} \\
      0 & 0 & 0 & 0 & c y_{12} \\
      0 & 0 & 0 & 0 & 0
      \end{smallmatrix}
      \right],
$$
\begin{equation}\label{eq:10}
y_{{23}} \left(y_{{34}}x_{{12}} - x_{{34}}y_{{12}}\right)\neq 0,
\end{equation}
and
$X_4= \left[
      \begin{smallmatrix}
      0 & 0 & a_{13} & a_{14} & a_{15} \\
      0 & 0 & 0 & a_{24} & a_{25} \\
      0 & 0 & 0 & 0 & -c a_{13}\frac{y_{34}}{y_{23}} \\
      0 & 0 & 0 & 0 &  0 \\
      0 & 0 & 0 & 0 & 0
      \end{smallmatrix}
      \right]$.
From
$[X_1, X_4] = [X_2, X_4]= 0$, we obtain the following equations
\begin{enumerate}[(a)]
\item $a_{{24}}x_{{12}}-a_{{13}}x_{{34}}= 0$,
\item $a_{{24}}y_{{12}}-a_{{13}}y_{{34}}= 0$,
\item $\left(a_{{25}}-c a_{{14}}\right)x_{{12}} - a_{{13}}\left(\frac {cx_{{13}}y_{{34}}}{y_{{23}}}-x_{{35}}\right)= 0$,
\item $\left(a_{{25}}-ca_{{14}}\right)y_{{12}} - a_{{13}}\left(\frac{cy_{{13}}y_{{34}}}{y_{{23}}}-y_{{35}}\right)= 0$.
\end{enumerate}
By Equations $(a) \text{ and } (b)$ and by Equation (\ref{eq:10}), we obtain
$a_{13}= a_{24}= 0$. Hence, by Equations $(c) \text{ and } (d)$ and by Equation (\ref{eq:10}) we get
$$X_4= \left[
      \begin{smallmatrix}
      0 & 0 & 0 & a_{14} & a_{15} \\
      0 & 0 & 0 & 0      &c a_{14} \\
      0 & 0 & 0 & 0 & 0 \\
      0 & 0 & 0 & 0 &  0 \\
      0 & 0 & 0 & 0 & 0
      \end{smallmatrix}
      \right].
$$

If
$c= 0$, by Lemma \ref{lema:2} in the case $(2)$ and Lemma \ref{lema:5} in the case $(1)$, we get
$$
X_1 = \left[
      \begin{smallmatrix}
      0 & x_{12} & x_{13} & x_{14} & x_{15} \\
      0 & 0 & x_{23} & x_{24} & x_{25} \\
      0 & 0 & 0 & x_{34} & x_{35} \\
      0 & 0 & 0 & 0 & 0\\
      0 & 0 & 0 & 0 & 0
      \end{smallmatrix}
      \right],
X_2 = \left[
      \begin{smallmatrix}
      0 & y_{12} & y_{13} & y_{14} & y_{15} \\
      0 & 0 & y_{23} & y_{24} & y_{25} \\
      0 & 0 & 0 & y_{34} & y_{35} \\
      0 & 0 & 0 & 0 & 0 \\
      0 & 0 & 0 & 0 & 0
      \end{smallmatrix}
      \right]
$$
\begin{equation}\label{eq:11}
\left(y_{{12}}x_{23} -x_{{12}}y_{{23}}\right)\left( x_{{34}}y_{{35}}-x_{{35}}y_{{34}} \right) \neq 0
\end{equation}
and
$
X_4= \left[
      \begin{smallmatrix}
      0 & 0 & a_{13} & a_{14} & a_{15} \\
      0 & 0 & 0 & a_{24} & a_{25} \\
      0 & 0 & 0 & 0 & 0\\
      0 & 0 & 0 & 0 &  0 \\
      0 & 0 & 0 & 0 & 0
      \end{smallmatrix}
      \right]
$. From
$[X_1, X_4] = [X_2, X_4]= 0$, we obtain the following equations
$$
\begin{cases}
\setlength{\unitlength}{13.10pt}
      \begin{picture}(20,1.2)
      \put(0.1,1.2){$x_{{12}}a_{{24}}-a_{{13}}x_{{34}}= 0$,}          \put(10,1.2){$x_{{12}}a_{{25}}-a_{{13}}x_{{35}}= 0$,}
      \put(0.1,0){$y_{{12}}a_{{24}}-a_{{13}}y_{{34}}= 0$,}        \put(10,0){$y_{{12}}a_{{25}}-a_{{13}}y_{{35}}= 0$.}
      \end{picture}
\end{cases}
$$
By Equation \ref{eq:11}, it is not difficult to see that
$$X_4= \left[
      \begin{smallmatrix}
      0 & 0 & 0 & a_{14} & a_{15} \\
      0 & 0 & 0 & 0      &0 \\
      0 & 0 & 0 & 0 & 0 \\
      0 & 0 & 0 & 0 &  0 \\
      0 & 0 & 0 & 0 & 0
      \end{smallmatrix}
      \right].
$$

\noindent Case $(2)$. Analysis similar to that in the Case $(1)$ with
$c = 0$, shows that
$$X_4= \left[
      \begin{smallmatrix}
      0 & 0 & 0 & 0 & a_{15} \\
      0 & 0 & 0 & 0      &a_{25} \\
      0 & 0 & 0 & 0 & 0 \\
      0 & 0 & 0 & 0 &  0 \\
      0 & 0 & 0 & 0 & 0
      \end{smallmatrix}
      \right],
$$ and the Proposition is proved.
\end{proof}

\begin{proof}$[\text{Proof of Proposition \ref{lema:4}.}]$
Let
$\h$ be the Lie algebra generated by
$\{X_1, X_2, X_3, Z_1, Z_2\}.$ By Remark \ref{remark:L59}, we have
$\h$ is a Lie subalgebra of
$\mathfrak{n}_5(\k)$ isomorphic to
$L_{5,9}$ and by Lemma \ref{lema:1}, we have the center
$\mathfrak{z}(\h)$ is of any of the following ways
\begin{enumerate}[(1)]
\item $\mathfrak{z}(\h)= \operatorname{span}_{\k}  \{E_{14} + cE_{25}, E_{15}\}$ with
      $c \in \k$;
\item $\mathfrak{z}(\h)= \operatorname{span}_{\k}  \{E_{25}, E_{15}\}$.
\end{enumerate}
We first assume the Case $(1)$. If
$c \neq 0$, by Lemma \ref{lema:2}(1) and Lemma \ref{lema:5}(1), we get
$$
X_1 = \left[
      \begin{smallmatrix}
      0 & x_{12} & x_{13} & x_{14} & x_{15} \\
      0 & 0 & \frac{y_{23}}{y_{34}}x_{34} & x_{24} & x_{25} \\
      0 & 0 & 0 & x_{34} & x_{35} \\
      0 & 0 & 0 & 0 & c x_{12} \\
      0 & 0 & 0 & 0 & 0
      \end{smallmatrix}
      \right],
X_2 = \left[
      \begin{smallmatrix}
      0 & y_{12} & y_{13} & y_{14} & y_{15} \\
      0 & 0 & y_{23} & y_{24} & y_{25} \\
      0 & 0 & 0 & y_{34} & y_{35} \\
      0 & 0 & 0 & 0 & c y_{12} \\
      0 & 0 & 0 & 0 & 0
      \end{smallmatrix}
      \right],
$$
\begin{equation}\label{eq:12}
y_{{23}} \left(y_{{34}}x_{{12}} - x_{{34}}y_{{12}}\right)\neq 0
\end{equation}
and
$X_4= \left[
      \begin{smallmatrix}
      0 & 0 & a_{13} & a_{14} & a_{15} \\
      0 & 0 & 0 & a_{24} & a_{25} \\
      0 & 0 & 0 & 0 & -c a_{13}\frac{y_{34}}{y_{23}} \\
      0 & 0 & 0 & 0 &  0 \\
      0 & 0 & 0 & 0 & 0
      \end{smallmatrix}
      \right].
$
From
$[X_2, X_4]= Z_1, [X_1, X_4]= \epsilon Z_2$, we obtain the following equations
\begin{enumerate}[(a)]
\item $a_{24}x_{12} - a_{13}x_{34} = -a_{13}x_{34} - a_{24}x_{12}$,
\item $a_{24}y_{12} - a_{13}y_{34} = -a_{13}y_{34} - a_{24}y_{12}$,
\item $a_{{13}}y_{{34}}= \left( x_{{12}}y_{{23}}-{\frac {y_{{12}}y_{{23}}x_{{34}}}{y_{{34}}}} \right) x_{{34}}$,
\item $a_{{13}}x_{{34}}= \epsilon \left( x_{{12}}y_{{23}}-{\frac {y_{{12}}y_{{23}}x_{{34}}}{y_{{34}}}} \right) y_{{34}}$.
\end{enumerate}
By $(a), (b)$ and Equation (\ref{eq:12}), we have
$a_{24}= 0$. By $(c), (d)$ and Equation (\ref{eq:12}), we get
$\epsilon = \frac{x_{34}^2}{y_{34}^2}$.

If
$c= 0$, by Lemma \ref{lema:2}(2) and Lemma \ref{lema:5}(1), we get
$$
X_1 = \left[
      \begin{smallmatrix}
      0 & x_{12} & x_{13} & x_{14} & x_{15} \\
      0 & 0 & x_{23} & x_{24} & x_{25} \\
      0 & 0 & 0 & x_{34} & x_{35} \\
      0 & 0 & 0 & 0 & 0\\
      0 & 0 & 0 & 0 & 0
      \end{smallmatrix}
      \right],
X_2 = \left[
      \begin{smallmatrix}
      0 & y_{12} & y_{13} & y_{14} & y_{15} \\
      0 & 0 & y_{23} & y_{24} & y_{25} \\
      0 & 0 & 0 & y_{34} & y_{35} \\
      0 & 0 & 0 & 0 & 0 \\
      0 & 0 & 0 & 0 & 0
      \end{smallmatrix}
      \right],$$
\begin{equation}\label{eq:13}
\left(y_{{12}}x_{23} -x_{{12}}y_{{23}}\right)\left( x_{{34}}y_{{35}}-x_{{35}}y_{{34}} \right) \neq 0
\end{equation} and
$
X_4= \left[
      \begin{smallmatrix}
      0 & 0 & a_{13} & a_{14} & a_{15} \\
      0 & 0 & 0 & a_{24} & a_{25} \\
      0 & 0 & 0 & 0 & 0\\
      0 & 0 & 0 & 0 &  0 \\
      0 & 0 & 0 & 0 & 0
      \end{smallmatrix}
      \right].
$
From
$[X_1, X_4] = \epsilon Z_2, [X_2, X_4]= Z_1$, we obtain the following equations
$$
\begin{cases}
x_{{12}}a_{{24}}-a_{{13}}x_{{34}}= \epsilon(y_{{12}} \left( x_{{23}}y_{{34}}-y_{{23}}x_{{34}} \right) - \left( x_{{12}}y_{{23}}-y_{{12}}x_{{23}}\right) y_{{34}}),\\
x_{{12}}a_{{25}}-a_{{13}}x_{{35}}= \epsilon( y_{{12}} \left( x_{{23}}y_{{35}}-y_{{23}}x_{{35}} \right) - \left( x_{{12}}y_{{23}}-y_{{12}}x_{{23}} \right) y_{{35}}),\\
y_{{12}}a_{{24}}-a_{{13}}y_{{34}}= x_{{12}} \left( x_{{23}}y_{{34}}-y_{{23}}x_{{34}} \right) - \left( x_{{12}}y_{{23}}-y_{{12}}x_{{23}}\right) x_{{34}},\\
y_{{12}}a_{{25}}-a_{{13}}y_{{35}}= x_{{12}} \left( x_{{23}}y_{{35}}-y_{{23}}x_{{35}}\right) - \left( x_{{12}}y_{{23}}-y_{{12}}x_{{23}} \right) x_{{35}}.
\end{cases}
$$
By Equation (\ref{eq:13}), we verify computationally that the only solution such that the set
$\{X_1, X_2, X_3, X_4, Z_1, Z_2\} \subseteq \mathfrak{n}_5(\k)$ is a linearly independent is if there exist
$a \in \k$ such that
$a^2 - \epsilon= 0$.

\noindent Case $(2)$. Analysis similar to that in the Case $(1)$ with
$c= 0$, shows that
$\{X_1, X_2, X_3, X_4, Z_1, Z_2\} \subseteq \mathfrak{n}_5(\k)$ is a linearly independent is if there exist
$a \in \k$ such that
$a^2 - \epsilon= 0$. It complete the proof.
\end{proof}

\begin{acknow}
This is part of the author's Ph.D. thesis written under supervision of  Professor Leandro Cagliero at the
University of C\'ordoba, Argentina.
\end{acknow}



\begin{thebibliography}{BML}



\bibitem[BNT]{BNT} J. C. Benjumea Acevedo, J. Nu{\~n}ez Valdes, A. F. Tenorio Villal\'on.
 \emph{Minimal Linear Representations of the Low-Dimensional Nilpotent Lie Algebras},
 Math. Scand. ,Vol. \textbf{102}, No. 1, (2008), 17-26.

\bibitem[Be]{Be} Y. Benoist.
 \emph{Une Nilvariete Non Affine},
  J. Diff. Geom., Vol. \textbf{41},(1995), 21-52.

\bibitem[B1]{BU1} D. Burde.
 \emph{Affine structures on nilmanifolds},
  J. Inter. Math., Vol. \textbf{7}(5),(1996), 599-616.

\bibitem[B2]{B} D. Burde.
\emph{A refinement of Ado's Theorem},
Archiv Math. Vol. \textbf{70}, (1998), 118-127.

\bibitem[B3]{B3} D. Burde.
\emph{Left-symmetric algebras, or pre-Lie algebras in geometry and physics},
Central European J. of Math. Vol. \textbf{4}(3), (2006), 323-357.

\bibitem[BG]{BG} D. Burde , F. Grunewald.
 \emph{Modules for certain Lie algebras of maximal class},
  J. Pure and Appl. Algebra, \textbf{99},(1995), 239-254.

\bibitem[BM]{BM} D. Burde, W. Moens.
 \emph{Minimal Faithful Representations of Reductive Lie Algebras},
 Archiv der Mathematik.,Vol. \textbf{89}, No. 6, (2007), 513-523.

\bibitem[CR]{CR} L. Cagliero, N. Rojas.
\emph{Faithful representation of minimal dimension of current Heisenberg Lie algebras},
Int. J. Math. Vol. \textbf{20} (11), (2009), 1347-1362.

\bibitem[FGH]{FGH} D. Fried, W. Goldman,  M. W. Hirsch.
\emph{Affine manifolds with nilpotent holonomy},
Comment. Math. Helv. \textbf{56} (1981), 487-523.

\bibitem[G]{Gr} W. A. De Graaf.
 \emph{Classification of 6-dimensional nilpotent Lie algebras over
       fields of characteristic not 2},
       J. of Alg. \textbf{309}, (2007), 640-653.

\bibitem[GN]{GN} W. de Graaf, W. Nickel.
\emph{Constructing faithful representations of finitely-generated
      torsion-free nilpotent groups},
 J. Symbolic Comput.,Vol. \textbf{33}, No. 1, (2002), 31-41.

\bibitem[J1]{J} N. Jacobson.
\emph{Lie Algebras},
Interscience Publishers, New York (1962).

\bibitem[J2]{J2} N. Jacobson.
 \emph{Schur's theorem on commutative matrices},
 Bull. Amer. Math. Soc.\textbf{50} (1944), 431-436.

\bibitem[K]{K} H. Kim.
 \emph{Complete left-invariant affine structures on nilpotent Lie groups},
 J. Dif. Geometry \textbf{24}(1986), 373-394.

\bibitem[KB]{KB} Y-F. Kang, C-M. Bai.
 \emph{Refinement of Ado's theorem in low dimensions and application in affine geomery},
 Communications in Algebra, \textbf{36}(1) (2008), 82-93.

\bibitem[M]{MM} M. Mirzakhani.
 \emph{A Simple Proof of a Theorem of Schur},
American Mathematical Monthly, Vol. \textbf{105}, No. 3,(1992), 260-262.

\bibitem[Mi]{Mi}  J. Milnor.
 \emph{On fundamental groups of complete affinely flat manifolds},
Adv. Math., \textbf{25} (1977), 178-187.

\bibitem[Mo]{Mo} V. V. Morozov.
 \emph{Classification of nilpotent Lie algebras of sixth order},
Izv. Vyss. Ucebn. Zaved. Mat. 1958 \textbf{4(5)} (1958), 161-171.

\bibitem[N]{N} W. Nickel.
 \emph{Matrix representations for torsion-free nilpotent groups by Deep Thought},
  J. of Algebra, \textbf{300},(2006), 376-383.

\bibitem[Ni]{Ni} O. A. Nielsen.
 \emph{Unitary Representations and Coadjoint Orbits of Low-Dimensional Nilpotent Lie Groups},
Queen's Papers in Pure and Appl. Math., \textbf{63}, Queen's University, Kingston, ON, 1983.

\bibitem[R]{R} N. Rojas.
 \emph{Minimal Faithful Representation of the Heisenberg Lie algebra with abelian factor},
 	arXiv:1206.5867v1 [math.RT]

\bibitem[S]{S} I. Schur.
 \emph{Zur Theorie vertauschbarer Matrizen},
 J. Reine Angew. Mathematik , \textbf{130} (1905), 66-76.

\bibitem[Se]{Se} D. Segal.
 \emph{Free left-symmetric algebras and an analogue of the Poincar\'e-Birkhoff-Witt Theorem},
    J. Algebra \textbf{164} (1994), 750-772.
\end{thebibliography}
\end{document}